\documentclass[a4paper, 11pt]{amsart}

\title{A pasting theorem for iterated Segal spaces}
\author{Jaco Ruit}
\address{Mathematisch Instituut, Universiteit Utrecht, The Netherlands}
\email{j.c.ruit@uu.nl}

\usepackage{amssymb}
\usepackage{amsfonts}
\usepackage{amsthm}

\usepackage{kpfonts}
\usepackage[mathscr]{eucal}

\usepackage[top=1.5in, bottom=1.5in, left=1.5in, right=1.5in]{geometry}

\usepackage{tikz-cd}{}  
\usetikzlibrary{backgrounds}
\usetikzlibrary{patterns} 
\usepackage{enumitem}
\usepackage{mathtools}
\usepackage[linktocpage]{hyperref}
\usepackage{aliascnt}
\usepackage{comment} 
\usepackage{todonotes}
\usepackage{adjustbox}
\usepackage{varioref}

\newtheorem{theorem}{Theorem}
\newcommand{\jnewtheorem}[2]{
	\newaliascnt{#1}{theorem}
	\newtheorem{#1}[#1]{#2}
	\aliascntresetthe{#1}
}
\numberwithin{theorem}{section}
\jnewtheorem{lemma}{Lemma}
\jnewtheorem{proposition}{Proposition}
\jnewtheorem{corollary}{Corollary}
\jnewtheorem{conjecture}{Conjecture}

\newtheorem{thmintro}{Theorem}

\theoremstyle{definition}
\jnewtheorem{definition}{Definition}
\jnewtheorem{example}{Example}
\jnewtheorem{remark}{Remark}
\jnewtheorem{question}{Question}
\jnewtheorem{fact}{Fact}
\jnewtheorem{construction}{Construction}

\labelformat{theorem}{Theorem #1}
\labelformat{lemma}{Lemma #1}
\labelformat{proposition}{Proposition #1}
\labelformat{corollary}{Corollary #1}
\labelformat{definition}{Definition #1}
\labelformat{example}{Example #1}
\labelformat{section}{Section #1}
\labelformat{subsection}{Subsection #1}
\labelformat{equation}{(#1)}
\labelformat{remark}{Remark #1}
\labelformat{exercise}{Exercise #1}
\labelformat{question}{Question #1}
\labelformat{chapter}{Chapter #1}
\labelformat{construction}{Construction #1}
\labelformat{conjecture}{Conjecture #1}
\labelformat{thmintro}{Theorem #1}

\hypersetup{
	pdftitle={A pasting theorem for iterated Segal spaces},
	pdfauthor={Jaco Ruit},
	bookmarksnumbered=true,     
	bookmarksopen=true,         
	bookmarksopenlevel=1,       
	colorlinks=true,      
	pdfencoding=unicode,      
	pdfstartview=Fit,           
	pdfpagemode=UseOutlines,    
	pdfpagelayout=OneColumn,
	linkcolor = {orange!80!black},
	urlcolor = {orange!80!black},
	citecolor = {cyan!60!black}
}

\newcommand{\N}{\mathbb{N}}
\newcommand{\cat}[1]{\textup{#1}}
\newcommand{\op}{\mathrm{op}}
\newcommand{\im}{\mathrm{im}}
\newcommand{\map}{\mathrm{Map}}
\DeclareMathOperator{\colim}{\mathrm{colim}}
\newcommand{\fun}{\mathrm{Fun}}
\newcommand{\Cat}{\cat{Cat}}

\renewcommand{\S}{\mathscr{S}}
\newcommand{\C}{\mathscr{C}}
\renewcommand{\D}{\mathscr{D}}

\newcommand{\Seg}{\mathrm{Seg}}
\newcommand{\Hom}{\mathrm{Hom}}
\newcommand{\Sp}{\mathrm{Sp}}

\newcommand{\Shape}{\cat{Shape}}

\newcommand{\dtr}{\mathrm{tr}}

\definecolor{pamblue}{rgb}{.78, .90, .98}

\begin{document}

\begin{abstract}
We introduce a novel notion of pasting shapes for iterated Segal spaces which classify particular arrangements of composing cells in $d$-uple Segal spaces. Using this formalism,  we then continue 
to prove a pasting theorem for these iterated Segal spaces. 
\end{abstract} 

\maketitle

\setcounter{tocdepth}{1}
\tableofcontents

\section{Introduction} 

A $d$-uple category is a generalization of a $d$-category, introduced by Ehresmann \cite{Ehresmann},  which has $k$-cells in $d\choose k$ different directions for $0\leq k \leq d$.
In such a structure, the compatible $k$-cells that point in the same direction 
may be composed, while cells with different directions may be related by $(k+1)$-cells. For $d=2$ and $d=3$, these are also known as \textit{double categories} \cite[Definition 10]{Ehresmann} and \textit{intercategories} \cite[Section 1]{GrandisPare2}, respectively.
These $d$-uple categories are the natural place to study mathematical structures that allow for multiple sorts of morphisms (1-cells) between them. 

For instance, in the realm of algebra, one can consider not only the usual maps between rings but also bimodules between rings. These are the two 
directions of 1-cells in the 
\textit{Morita double category}  (\cite[Example 2.3]{Shulman}). Since its introduction, the theory of double categories has found a wide range of applications throughout category theory. 
For instance, the theory of 2-categorical limits admits a neat description using double categories \cite{GrandisPare}.
Pseudo 2-functors that form a \textit{proarrow equipment} are better understood as being double categories with additional properties \cite{Verity}, \cite{Shulman2}. 
These proarrow equipments give rise to formal category theories. 
In \cite{Shulman}, we see an application of the theory of double categories to homotopy theory.

Whereas in a usual category,  
one may consider strings of compatible morphisms and take their composites, there are now many configurations of compatible cells in a $d$-uple category. 
For instance, one may start with compatible 2-cells
$v_1, v_2, v_3, v_4, v_5, v_6$ in a double category $D$:
\[
	\begin{tikzcd}
		x_{00}\arrow[d] \arrow[r, ""'name=f1] & x_{10} \arrow[d] \arrow[r, ""'name=v2] & x_{20} \arrow[rr, ""'name=f2] &  & x_{40}\arrow[d] \\
		x_{01} \arrow[r, ""name=t1]\arrow[dd] & |[""name=f3]|x_{11}\arrow[r] &x_{21}\arrow[dd] \arrow[r, ""'name=f4] & x_{31} \arrow[d] \arrow[r, ""'name=f5] & x_{41}\arrow[d] \\
		& & & x_{32}\arrow[d]\arrow[r, ""name=f6] & x_{42} \arrow[d]\\
		x_{03}\arrow[rr, ""name=t3] & & x_{23} \arrow[r,""name=t4] &x_{33}\arrow[r,""name=t6] & x_{43}, 
		\arrow[from=f1,to=t1, Rightarrow, "v_1"] 
		\arrow[from=f4,to=t4, Rightarrow, "v_4"]
		\arrow[from=f3,to=t3, Rightarrow, "v_3", shorten <= 14pt]
		\arrow[from=f2,to=f4, Rightarrow, "v_2", start anchor={[xshift=-23pt]}, shorten >= 5pt]
		\arrow[from=f5,to=f6, Rightarrow, "v_5"]
		\arrow[from=f6,to=t6, Rightarrow, "v_6", shorten <= 5pt]
	\end{tikzcd}
\]
and wonder: does there exist a unique composite 2-cell $v$ in $D$? It has been shown by Dawson and Par\'e \cite{DawsonPare} that (in particular) this arrangement of 2-cells admits 
such a composite $v$. However, not every compatible arrangement 
of 2-cells in a general double category admits a composite. The authors of \cite{DawsonPare} established that there is an arrangement that does not have a composite in a general double category, which is called the \textit{pinwheel} (we will see this arrangement 
again in \ref{ssection.cpshapes}), and which is in a particular sense, the canonical example of such an inadmissible arrangement (see \cite{Dawson}). 
The procedure of obtaining new cells by composing compatible arrangements is also called \textit{pasting}, which was first introduced by B\'enabou \cite{Benabou}
in the context of 2-categories. 

The practice of pasting cells occurs in the context of many different categorical structures. Hence, one would like to have access to a \textit{pasting theorem}  
that asserts the existence and uniqueness of composites for certain configurations of cells in the categorical structures one considers. 
A famous pasting theorem for 2-categories was formulated and proven by Power \cite{Power}.  Nowadays, a wide range of pasting theorems for (strict) $\omega$-categories are available. Forest \cite{Forest} has recently 
unified the main pasting theorems in this context to a more general pasting theorem for  $\omega$-categories. There 
is a pasting theorem for double categories due to Dawson and Par\'e \cite{DawsonPare}.

The emergence of (weak) $\infty$-category theory has created a
need for variants of pasting theorems in the weaker setting. In this context, it is no longer natural to ask for unique composites, 
but instead require that the space\footnote{By space, we will always mean an $\infty$-groupoid in this paper.} of composites is contractible.  Lately, Hackney, Ozornova, Riehl and Rovelli \cite{Infty2Pasting} have proven such a pasting 
theorem for $(\infty,2)$-categories that generalizes Power's pasting theorem for 2-categories. 
There is an $\infty$-analog for double categories as well, so called \textit{double $\infty$-categories}, introduced by Haugseng \cite{HaugsengPhD} and further  
studied by Moser \cite{Moser}.
Thus, one may now ask: does there exist a pasting theorem for these double $\infty$-categories?

The goal of this paper is to answer this question affirmatively. We will treat a more general pasting problem for $d$-uple categories in the weak $\infty$-categorical setting. 
Note that this would already be interesting from the strict perspective: 
the pasting theorem particularly yields a novel pasting theorem for $d$-uple categories. The author does not know of an existing similar result for $d > 2$. 

We will consider an $\infty$-categorical variant of $d$-uple categories that are called \textit{$d$-uple Segal spaces} \cite{Haugseng2}.
Ordinary (1-uple) Segal spaces were first introduced by Rezk  \cite{RezkSeg} as a model for $\infty$-categories.
The $d$-uple Segal spaces are iterated variants of these that have additional directions for morphisms. For instance, a $2$-uple Segal space $X$ contains a space of objects, and between any two objects, 
a space of \textit{vertical} and \textit{horizontal} arrows. Thus a 2-uple Segal space has two categorical directions instead of merely one as is the case for an ordinary 
Segal space. Compatible arrows of $X$ that have the same direction, can be composed in a coherently associative fashion. Moreover, $X$ contains a space of 2-cells. A 2-cell may be pictured as a square
\[
	\begin{tikzcd}
		a \arrow[r, ""'name=f]\arrow[d] & b \arrow[d] \\
		c \arrow[r, ""name=t] & d \arrow[from=f, to=t, Rightarrow]
	\end{tikzcd}
\]
in $X$. Here the arrows that point horizontally are horizontal arrows of $X$ and similarly for the vertical ones. That is, 2-cells have a source and target vertical arrow and a source and target horizontal arrow.
Again, $X$ has a coherently associative composition for these 2-cells, which is compatible with the composition of 1-cells.
In general, $d$-uple Segal spaces contain $d$ categorical directions, which may interact coherently using higher cells.  

These $d$-uple Segal spaces are rich structures that 
play a useful role in $\infty$-category theory. For instance, by `truncating' all but one of the categorical directions, they can be used to model $(\infty,d)$-categories. This 
is the model for $(\infty,d)$-categories due to Barwick \cite{BarwickPhD}, which we will briefly discuss in \ref{section.outlook}. It has been (directly) compared to other 
models for $(\infty,d)$-categories by Bergner and Rezk in \cite{BergnerRezk1} and \cite{BergnerRezk2}, and Loubaton, Ozornova and  Rovelli in \cite{Loubaton} and \cite{OzornovaRovelli}. 
Consequently, $d$-uple Segal spaces may act as a useful intermediate step towards constructing $(\infty,d)$-categories. For instance, in order to construct \textit{$(\infty,d)$-categories of iterated 
spans} or to construct the \textit{Morita $(\infty,d)$-categories} it is more convenient to define their encompassing $d$-uple Segal spaces first (see \cite{Haugseng} and \cite{Haugseng2}). 
We also believe that $d$-uple Segal spaces can be used to study phenomena of $(\infty,d)$-categories. For instance, the universal property of the $(\infty,2)$-category 
of spans \cite{ElmantoHaugseng} should be the shadow of a universal property of $2$-uple Segal space of spans that is 
analogous to one in the strict case \cite{DawsonParePronk}. We hope to study this in future work.

This work grew out of the author's study of double $\infty$-categories, which can be viewed 
as 2-uple Segal spaces that satisfy a \textit{completeness} assumption (see \ref{rem.d-uple-infty-cats}). More specifically, out 
of the interest in a particular class of double $\infty$-categories: so-called \textit{$\infty$-equipments}, which are generalizations of proarrow equipments to the $\infty$-categorical context. 
These $\infty$-equipments  offer a context for \textit{synthetic} or \textit{formal category theory}. 
There exist suitable $\infty$-equipments for equivariant, indexed, internal, fibered, enriched and ordinary $\infty$-category theory. We commenced the study of these $\infty$-equipments in \cite{Equip1}.

\subsection*{Content of the paper} We will commence the paper by setting up the necessary preliminaries, including revisiting the definition of iterated Segal spaces, in \ref{section.prelims}.

Subsequently, in \ref{section.pshapes}, we introduce the novel notion of \textit{$d$-dimensional pasting shapes} whose \textit{nerves} classify arrangements of cells in a $d$-uple Segal space. 
 For $d=3$, these include arrangements of rectangular cuboids, whose faces may be subdivided into smaller rectangles.
The faces of these rectangles may in turn be subdivided into smaller edges. We will also discuss  
some fundamental properties that pasting shapes may have. In particular, we single out an important class 
of pasting shapes: the so-called \textit{(locally) composable} ones. The composable 
pasting shapes are certain well-behaved pasting shapes. For
$d=2$, the aforementioned pinwheel of Dawson and Par\'e is an example of a pasting shape that is not composable.  

After setting up the theory of pasting shapes, we can formulate the main result of this paper, the \textit{pasting theorem} for $d$-uple Segal spaces:
\begin{thmintro}[The pasting theorem]\label{thmintro.pasting-theorem}
	Suppose that $I_1, \dotsc, I_n$ is a covering of a $d$-dimensional locally composable pasting shape $I$. Then $I$ can be written as a union
	$$
	I = \bigcup_{i=1}^n I_i,
	$$
	and this colimit description is preserved when passing to the $\infty$-category of $d$-uple Segal spaces via the nerve functor for pasting shapes.
\end{thmintro}
\noindent The precise statement appears as \ref{thm.pasting-theorem} in the paper.
This theorem can be used to show \ref{cor.spine-inclusion}, which states that the \textit{spine inclusion} associated with a composable pasting shape is an equivalence.
The composable pasting shapes thus classify arrangements of cells that admit a composite that is unique up to contractible choice (see \ref{cor.composite-contractible-choice}). 

The technical heart of the paper is in \ref{section.proof}. Here, 
we proceed by induction on the dimension of pasting shapes, 
to give a proof of \ref{thmintro.pasting-theorem}.
After this demonstration, we conclude the paper in \ref{section.outlook} by giving an idea of how our pasting theorem for $d$-uple Segal spaces may yield 
a pasting theorem for $(\infty,d)$-categories.

\subsection*{Conventions}
We will use the language of $\infty$-categories throughout this article. 
For definiteness, we work with the model of quasi-categories for $\infty$-categories as developed by Joyal and Lurie \cite{HTT}. The readers that prefer to use the language of model categories, may interpret the main results of this paper within the appropriate model categories associated with the $\infty$-categories in question. We will use 
the following customary notation: 
\begin{itemize}
	\item The (large) $\infty$-categories of spaces \cite[Definition 1.2.16.1]{HTT} and $\infty$-categories \cite[Definition 3.0.0.1]{HTT} are denoted by $\S$ and $\Cat_\infty$, respectively. 
	\item If $\C$ is an $\infty$-category, we write 
	$$
	\map_\C(-,-) : \C^\op \times \C \rightarrow \S
	$$
	for its associated mapping space functor \cite[Subsection 5.1.3]{HTT}.
	\item Every $1$-category $\C$ will be viewed as an $\infty$-category, suppressing the notation of the nerve \cite[Subsection 1.1.2]{HTT}.
\end{itemize}
\subsection*{Acknowledgements}
I want to thank my PhD-supervisor, Lennart Meier, for the helpful conversations during the writing of this paper and his useful comments on the draft versions.
Furthermore, I would like to express my gratitude to the anonymous referee for the useful suggestions that greatly improved the presentation of this article.

During the writing of this paper, the author was funded by the Dutch Research Council (NWO) through the grant ``The interplay of orientations and symmetry'', grant no. OCENW.KLEIN.364.

\section{Preliminaries}\label{section.prelims}

A clear definition of the (strict) $d$-uple categories introduced by Ehresmann \cite{Ehresmann} can be given in terms of 
Grothendieck's notion of \textit{categorical objects} \cite{Grothendieck}. 
Following \cite{Haugseng2} and \cite{Haugseng}, we will introduce an $\infty$-categorical variant on $d$-uple categories via the same principle. 

\subsection{Categorical objects}\label{ssection.cat-objs}

To this end, we first need the notion of categorical objects in an $\infty$-category. 
Throughout this section, we will fix an $\infty$-category $\C$ and assume that it has all pullbacks.
The following definition is due to \cite{LurieInfty2}:

\begin{definition}\label{def.cat-objects}
A \textit{categorical object $X$ in $\C$} is a simplicial object $X : \Delta^\op \rightarrow \C$
such that the so-called \textit{Segal map}
$$
X([n]) \rightarrow X(\{0 \leq 1\}) \times_{X(\{1\})} \dotsb \times_{X(\{n-1\})} X(\{n-1\leq n\})
$$ 
is an equivalence for all $n$. The full subcategory of $\fun(\Delta^\op, \C)$ spanned by the categorical objects in $\C$ is denoted by $\Cat(\C)$. 
\end{definition}

\begin{example} We have the following examples:
	\begin{itemize}
	\item A categorical object in the $(2,1)$-category of categories is a pseudo double category. These were defined by Grandis and Par\'e \cite{GrandisPare}.
	\item If $\C$ is given by the $\infty$-category of spaces $\S$, then $\Cat(\S)$ is the $\infty$-category that underlies 
	\textit{the model category of Segal spaces}, constructed by Rezk in \cite[Theorem 7.1]{RezkSeg}. Henceforth, we will refer to categorical 
	objects in $\S$ as \textit{Segal spaces}.
	\item A categorical object in $\Cat_\infty$ is called a \textit{double $\infty$-category}. These were first studied by Haugseng in \cite{HaugsengPhD}.
	\end{itemize}
\end{example}

Note that the subcategory 
$\Cat(\C) \subset \fun(\Delta^\op, \C)$ is closed under limits, so that limits of categorical objects may be computed pointwise.
In particular, we deduce that $\Cat(\C)$ again admits all pullbacks. Thus, we may iterate \ref{def.cat-objects}:

\begin{definition}
We define the $\infty$-category of \textit{$d$-uple categorical objects in $\C$} by 
$$
\Cat^d(\C) := \Cat(\dotsb\Cat(\Cat(\C))\dotsb).
$$
\end{definition} 

By adjunction, the $\infty$-category of $d$-uple categorical objects in $\C$ can be described 
as a full subcategory of the $\infty$-category $$ \fun(\Delta^{\op, \times d}, \C)$$
of \textit{$d$-uple simplicial objects in $\C$}. In the case that $\C$ is presentable, then
 $\Cat^d(\C)$ is a (left) reflective subcategory of $\fun(\Delta^{\op, \times d}, \C)$. I.e.\ the  
 inclusion admits a left adjoint (see also \cite[Remark 5.2.7.9]{HTT}) in this case. We will demonstrate how this can be established.

\begin{construction} Let us write $$\Delta[n_1, \dotsc, n_d] : \Delta^{\op, \times d} \rightarrow \cat{Set} \subset \S$$ for the discrete presheaf represented by $([n_1] , \dotsc, [n_d]) \in \Delta^{\times d}$.
We consider the category $\mathbb{G}^{\times d}/([n_1], \dotsc, [n_d])$ defined by the pullback square 
\[
	\begin{tikzcd}
		\mathbb{G}^{\times d}/([n_1], \dotsc, [n_d]) \arrow[r]\arrow[d] & \Delta^{ \times d}_\mathrm{in}/([n_1], \dotsc, [n_d]) \arrow[d] \\
		\mathbb{G}^{\times d} \arrow[r] & \Delta_{\mathrm{in}}^{ \times d}.
	\end{tikzcd}
\]
Here $\Delta_{\mathrm{in}}$ is the subcategory of $\Delta$ spanned by the \textit{inert} maps: i.e.\ maps that are injective and have convex image. 
The category $\mathbb{G}$ is in turn the full 
subcategory of $\Delta_{\mathrm{in}}$ spanned by $[0]$ and $[1]$, the free 0- and 1-cells.
The \textit{spine inclusion} for $([n_1], \dotsc, [n_d])$ is now defined to be the map of presheaves
$$i_{n_1, \dotsc, n_d} : \colim_{([k_1], \dotsc, [k_d]) \in \mathbb{G}^{\times d}/([n_1], \dotsc, [n_d])} \Delta[k_1, \dotsc, k_d] \rightarrow \Delta[n_1, \dotsc, n_d].$$
\end{construction}

We will see a more geometrically flavored definition of the spine inclusions in \ref{section.pshapes}.

\begin{proposition}
	If $\C$ is presentable, then there exists a left adjoint 
	$$
	L : \fun(\Delta^{\op, \times d}, \C) \rightarrow \Cat^d(\C),
	$$
	to the inclusion 
	$\Cat^d(\C) \rightarrow \fun(\Delta^{\op, \times d}, \C)$.
\end{proposition}
\begin{proof}
	Let $G$ be a (small) set of compact generators for $\C$, so that any object of $\C$ can be written as the colimit of a small diagram of objects in $G$ (see characterization (6) of \cite[Theorem 5.5.1.1]{HTT}). Then we obtain a set $S$ of maps of $d$-uple simplicial $\C$-objects whose 
elements are given by
$$
i_{n_1, \dotsc, n_m} \otimes x : \Delta^{\op, \times d} \xrightarrow{(i_{n_1, \dotsc, n_m}, x)}\S \times \C \xrightarrow{-\otimes -} \C, \quad n_1, \dotsc, n_m \in \N, x \in G.
$$
 Here $- \otimes -$ denotes the canonical tensor product $\S \times \C \rightarrow \C$ that is determined 
by fitting in adjunctions
$$
- \otimes x : \S \rightleftarrows  \C : \map_\C(x,-),
$$ 
for $x \in \C$. This tensor product is discussed in \cite[Subsection 4.4.4]{HTT}. By construction, 
the $S$-local objects (see \cite[Definition 5.5.4.1]{HTT}) are precisely the $d$-uple categorical objects in $\C$. Hence, the result follows from \cite[Proposition 5.5.4.15]{HTT}.
\end{proof}

It directly follows from the definition that $d$-uple categorical objects are functorial in the following way:
\begin{proposition}\label{prop.cat-objs-funct}
	Suppose that $f : \C \rightarrow \D$ is a finite limit preserving functor between $\infty$-categories that admit all finite limits.
	Then the induced functor $\fun(\Delta^{\op, \times d}, \C) \rightarrow \fun(\Delta^{\op, \times d}, \D)$ restricts 
	to a functor $\Cat^d(\C) \rightarrow \Cat^d(\D)$.
\end{proposition}

Moreover, there is functoriality in the dimension $d$. Fix $0 \leq k \leq d$ and a choice of 
a $k$-tuple of indices $1 \leq h_1 < \dotsc < h_k \leq d$. The category $\Delta^\op$ has an initial object 
given by $[0]$. Hence, the projection functor 
$$
p_h : \Delta^{\op, \times d} \rightarrow \Delta^{\op, \times k} : ([n_1], \dotsc, [n_d]) \mapsto ([n_{h_1}],\dotsc,[n_{h_k}])
$$
admits a (fully faithful) left adjoint $j_h : \Delta^{\op, \times k} \rightarrow \Delta^{\op, \times d}$ 
given on objects by 
$$
j_h([m_1], \dotsc, [m_k])_a = \begin{cases}
	[m_b] & \text{if $a = h_b$},  \\
	[0] & \text{otherwise}.
\end{cases}
$$
We now obtain an induced adjunction
$$p_h^* : \fun(\Delta^{\op, \times k}, \C) \rightleftarrows \fun(\Delta^{\op, \times d}, \C) : j_h^*$$ 
on functor $\infty$-categories.
The following is readily verified:

\begin{proposition}\label{prop.cat-k-d-incl}
	The adjunction $(p_h^*, j_h^*)$  restricts to an adjunction 
	$$p_h^* : \Cat^k(\C) \rightleftarrows \Cat^d(\C) : j_h^*.$$
\end{proposition}

\subsection{Iterated Segal spaces} We now specialize the notions of iterated categorical objects to obtain a variant on $d$-uple categories (cf.\ \cite[Section 3]{Haugseng2}):

\begin{definition}
	The $d$-uple categorical objects in the $\infty$-category $\S$ of spaces  and in the 
	category $\cat{Set}$ of sets are called \textit{$d$-uple Segal spaces} and \textit{$d$-uple Segal sets}, respectively.
\end{definition}

The (limit-preserving) inclusion $\cat{Set} \rightarrow \S$ gives rise to an inclusion 
$$
\fun(\Delta^{\op, \times d}, \cat{Set}) \rightarrow \fun(\Delta^{\op, \times d}, \S)
$$
that further restricts to a fully faithful functor
$$
\Cat^d(\cat{Set}) \rightarrow \Cat^d(\S)
$$ 
on account of \ref{prop.cat-objs-funct}. We leave these inclusions implicit and view $d$-uple Segal sets 
and $d$-uple simplicial sets as (discrete) $d$-uple Segal spaces and $d$-uple simplicial spaces, respectively.

\begin{remark}\label{rem.d-uple-infty-cats}
	Following the definition of ordinary multifold categories, a \textit{$d$-uple $\infty$-category} should be a $(d-1)$-uple categorical object in $\Cat_\infty$.
	Joyal and Tierney \cite{JoyalTierney} have shown that $\Cat_\infty$ is a reflective subcategory of $\Cat(\S)$, 
	identifying $\infty$-categories with the Segal spaces that are \textit{complete}   (see \ref{section.outlook} for a definition and a precise statement of this fact). 
	This implies that 
	$\Cat^{d-1}(\Cat_\infty)$ can be viewed as a reflective subcategory of $\Cat^d(\S)$ spanned by 
	those $d$-uple Segal spaces that satisfy a certain additional completeness condition. Throughout this article, we prefer to work with the latter, more general, notion.
\end{remark}

\begin{definition}
	We say that a map $f : X \rightarrow Y$ between $d$-uple simplicial spaces is a \textit{Segal equivalence} if it is carried to an equivalence by the localization functor $L : \fun(\Delta^{\op, \times d}, \S) \rightarrow \Cat^d(\S)$, or, equivalently if the induced map 
	$$
	f^* : \map_{\fun(\Delta^{\op, \times d}, \S)}(Y, Z) \rightarrow \map_{\fun(\Delta^{\op, \times d}, \S)}(X, Z)
	$$
	of spaces is an equivalence for every $d$-uple Segal space $Z$.
\end{definition}

\begin{remark}
	On account of \cite[Proposition 5.4.15]{HTT}, the class of Segal equivalences is the \textit{strongly generated class of morphisms} (see \cite[Definition 5.5.4.5]{HTT}) generated 
	by all spine inclusions $i_{n_1, \dotsc, n_d}$, $n_1, \dotsc, n_d \geq 0$.
\end{remark}

We end this subsection with two important propositions that will be used repeatedly throughout this article to identify Segal equivalences between $d$-uple simplicial sets.

\begin{proposition}\label{prop.pushout-seq}
	Suppose that we have a pushout diagram 
	\[
		\begin{tikzcd}
			A \arrow[r] \arrow[d] & X \arrow[d] \\
			B \arrow[r] & Y
		\end{tikzcd}
	\]
	in $\fun(\Delta^{\op, \times d}, \cat{Set})$ so that the map $A \rightarrow B$ is injective and a Segal equivalence. Then the map $X \rightarrow Y$ is injective and a Segal equivalence as well.
\end{proposition}
\begin{proof}
	Since the localization functor $L$ is a left adjoint, we deduce that Segal equivalences are closed under pushouts in $\fun(\Delta^{\op, \times d}, \S)$. 
	Hence, it suffices to show that the pushout square 
	is preserved under the functor $$\fun(\Delta^{\op, \times d}, \cat{Set}) \rightarrow \fun(\Delta^{\op, \times d}, \S).$$ 
	This can be checked level-wise. The inclusion $\cat{Set} \rightarrow \S$ factors as $$\cat{Set} \rightarrow \cat{sSet} \rightarrow \S,$$ where the functor $\cat{sSet} \rightarrow \S$ 
	witnesses the $\infty$-category of spaces $\S$ as the localization of the category of simplicial sets $\cat{sSet}$ at the weak homotopy equivalences, i.e.\ it witnesses $\S$ as the underlying $\infty$-category
	of $\cat{sSet}$ equipped with the Kan-Quillen model structure. This localization functor carries homotopy colimits to colimits (see \cite[Theorem 4.2.4.1]{HTT}). 
	Hence, we have to show that the square 
	\[
		\begin{tikzcd}
			A \arrow[r] \arrow[d] & X \arrow[d] \\
			B \arrow[r] & Y
		\end{tikzcd}
	\]
	is level-wise a homotopy pushout square in the Kan-Quillen model structure on $\cat{sSet}$. But this follows 
	from the fact that $A \rightarrow B$ is a level-wise cofibration, and all objects in $\cat{sSet}$ are cofibrant (see \cite[Proposition A.2.4.4]{HTT} for instance).
\end{proof}

\begin{proposition}\label{lemma.unions-seq}
	Let $X$ be a $d$-uple simplicial set. Suppose that $S$ is 
	a collection
	of pairs $(A,B)$ where $A \subset B \subset X$ are $d$-uple simplicial subsets, which 
	satisfies the following two conditions:
	\begin{enumerate}
		\item for every $(A,B)$, $(A',B') \in S$, we have  $A \cap B' \subset A'$ and $(A\cap B', B \cap B') \in S$,
		\item for every $(A,B) \in S$, the inclusion $A\rightarrow B$ is a Segal equivalence.
	\end{enumerate}
	Then if $(A_1, B_1),\dotsc, (A_n, B_n)$ are pairs in $S$, the inclusion
	$$
	\textstyle \bigcup_{i=1}^n A_i \rightarrow \bigcup_{i=1}^n B_i
	$$
	is a Segal equivalence as well.
\end{proposition}
\begin{proof}
	We proceed by induction on the length of the sequence. If $n = 0$, then there is nothing to show. Suppose that the statement 
	holds for sequences of length $n-1$. Suppose that $(A_1, B_1),\dotsc, (A_n, B_n)$ are pairs in $S$. 
	Then we can consider the factorization
	$$
	\textstyle \bigcup_{i=1}^n A_i \rightarrow B_1 \cup \bigcup_{i=2}^n A_i \rightarrow \bigcup_{i=1}^n B_i.
	$$
	Note that the right map fits in a pushout square 
	\[
		\begin{tikzcd}
		\bigcup_{i=2}^n A_i \cap B_1 \arrow[r]\arrow[d] &  B_1 \cup \bigcup_{i=2}^n A_i  \arrow[d] \\
		\bigcup_{i=2}^n B_i \cap B_1 \arrow[r] & \bigcup_{i=1}^n B_i.
		\end{tikzcd}	
	\]
	In light of the property (1) of $S$ and the induction hypothesis, the left vertical map is a Segal equivalence. 
	Hence, the map on the right is a Segal equivalence as well on account of \ref{prop.pushout-seq}.
	It thus suffices to show that the map 
	$\bigcup_{i=1}^n A_i \rightarrow B_1\cup \bigcup_{i=2}^n A_i$ 
	is a Segal equivalence. This fits in a pushout square 
	\[
		\begin{tikzcd}
		\bigcup_{i=1}^n A_i \cap B_1 \arrow[r]\arrow[d] &  \bigcup_{i=1}^n A_i  \arrow[d] \\
		B_1 \arrow[r] &  B_1\cup \bigcup_{i=2}^n A_i.
		\end{tikzcd}	
	\]
	On account of property (1), the left map is the inclusion $A_1 \rightarrow B_1$, which is a Segal equivalence, so that another application of \ref{prop.pushout-seq} yields 
	the desired result.
\end{proof}

\section{Pasting shapes and their nerves}\label{section.pshapes}

We commence this section by introducing \textit{pasting shapes} and their \textit{nerves}. The remaining of this section will be devoted to studying basic properties 
of these pasting shapes and introducing the pasting theorem.

\begin{definition}
A  \textit{$(d,k)$-box} (resp.\ \textit{non-degenerate $(d,k)$-box}) is a pair $$(x,y) \in \N^{\times d} \times \N^{\times d}$$ 
such that $x_a = y_a$ for $d-k$ indices $a$, and $x_a \leq y_a$ (resp.\ $x_a < y_a$) for the remaining $k$ indices $a$. We denote the set of $(d,k)$-boxes by $B^{d,k}$.

A $(d,k-1)$-box $(x',y')$ is called a \textit{face} of a $(d,k)$-box $(x,y)$ if $x'_a, y'_a \in \{x_a,y_a\}$ for all indices $a$.

We say that two (non-degenerate) $(d,k)$-boxes $(x,y)$, $(x', y')$ are \textit{adjacent} 
if $x_a = y_a = x'_a = y'_a$ for $d-k$ indices $a$, and out of the remaining $k$ indices $a_1, \dotsc, a_k$, there is at most one index $i$ such that $x_{a_i}<y_{a_i}=x'_{a_i} < y'_{a_i}$ and for the other indices $j$, we have $x_{a_j} = x'_{a_j} < y'_{a_j} = y_{a_j}$. 
In this case $(x,y')$ is a non-degenerate $(d,k)$-box, which we will call the \textit{join} of $(x,y)$ and $(x',y')$.

\end{definition}

\begin{definition}
A \textit{$d$-dimensional pasting shape $I$} is a subset $$B^d(I) \subset B^{d,d}$$ of $(d,d)$-boxes
giving rise to a chain 
$$
B^0(I) \subset B^1(I) \subset \dotsb \subset B^d(I), 
$$
where $B^k(I) := B^{d,k} \cap B^d(I)$ are called the \textit{$k$-boxes in $I$},
satisfying the following two properties:
\begin{enumerate}
	\item \textit{closure under faces:} for every box $(x,y)$ in $I$, the boxes $(x',y')$ such that $x'_a,y'_a \in \{x_a, y_a\}$ for each index $a$, are in $I$,
	\item \textit{closure under joins:} if $(x,y)$ and $(x',y')$ are adjacent $k$-boxes in $I$, then its join $(x,y')$ must be contained in $I$.
\end{enumerate}
Note that $B^0(I)$ is the diagonal of a subset $V(I) \subset \N^{\times d}$, which we will call the \textit{vertices} of $I$. Any box $(x,y)$ in $I$ has $x,y \in V(I)$.

A \textit{map of pasting shapes} $f : I \rightarrow J$ between $d$-dimensional pasting shapes $I$ and $J$ is a map $f : V(I) \rightarrow V(J)$ between their underlying sets of vertices, 
such that for any box $(x,y)$ in $I$, $(f(x), f(y))$ is a box in $J$ with the property that 
$f(x)_a = f(y)_a$ whenever $x_a = y_a$. With these maps, the $d$-dimensional pasting shapes form a category which we will denote by
$\Shape^d$.
\end{definition}

\begin{example}
The $1$-dimensional pasting shapes may be identified with subposets of $\N$.
\end{example}

\begin{example}\label{ex.pshapes-graph}
We may view $d$-dimensional pasting shapes as being particular graphs equipped with extra markings. For instance, for $d=2$, the following graph 
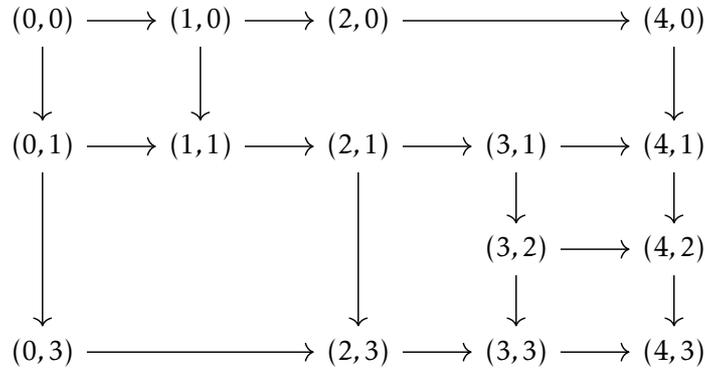
\[
	\begin{tikzcd}
		(0,0)\arrow[d] \arrow[r, ""'name=f1] & (1,0) \arrow[d] \arrow[r, ""'name=v2] & (2,0) \arrow[rr, ""'name=f2] &  & (4,0)\arrow[d] \\
		(0,1) \arrow[r, ""name=t1]\arrow[dd] & |[""name=f3]|(1,1)\arrow[r] & (2,1)\arrow[dd] \arrow[r, ""'name=f4] & (3,1) \arrow[d] \arrow[r, ""'name=f5] & (4,1)\arrow[d] \\
		& & & (3,2)\arrow[d]\arrow[r, ""name=f6] & (4,2) \arrow[d]\\
		(0,3)\arrow[rr, ""name=t3] & & (2,3) \arrow[r,""name=t4] &(3,3)\arrow[r,""name=t6] & (4,3)
	\end{tikzcd}
\]
depicts the smallest $2$-dimensional pasting shape $I'$ whose vertices are given by the vertices in the graph, and contains every 1-box $(x,y)$ for which
there exists a (directed) path of edges between $x$ and $y$. Note that the resulting pasting shape has no non-degenerate 2-boxes.
To remedy this, our picture needs to reflect which 2-boxes need to be included in the associated pasting shape. 

In this article, we will use the convention to color the backgrounds of the 2-boxes we would like to include. Thus the following picture 
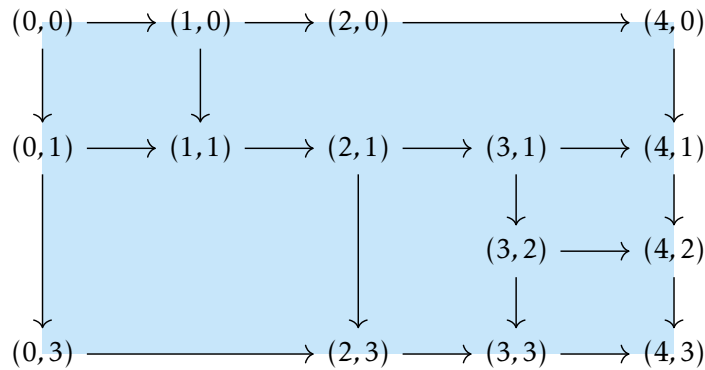
\[
	\begin{tikzcd}[execute at end picture={
		\scoped[on background layer]
		\fill[pamblue, opacity = 0.5] (a.center) -- (b.center) -- (d.center) -- (c.center);
	}]
		|[alias=a]|(0,0)\arrow[d] \arrow[r, ""'name=f1] & (1,0) \arrow[d] \arrow[r, ""'name=v2] & (2,0) \arrow[rr, ""'name=f2] &  &|[alias=b]|(4,0)\arrow[d] \\
		(0,1) \arrow[r, ""name=t1]\arrow[dd] & |[""name=f3]|(1,1)\arrow[r] & (2,1)\arrow[dd] \arrow[r, ""'name=f4] & (3,1) \arrow[d] \arrow[r, ""'name=f5] & (4,1)\arrow[d] \\
		& & & (3,2)\arrow[d]\arrow[r, ""name=f6] & (4,2) \arrow[d]\\
		|[alias=c]|(0,3)\arrow[rr, ""name=t3] & & (2,3) \arrow[r,""name=t4] &(3,3)\arrow[r,""name=t6] & |[alias=d]|(4,3)
	\end{tikzcd}
\]
depicts the $2$-dimensional pasting shape $I$ which contains $I'$ and all the non-degenerate 2-boxes whose faces are in $I$ and enclose a colored region in the picture. In this case, 
these are precisely all non-degenerate 2-boxes whose faces are in $I'$. In \ref{ex.boxdot}, we see an example 
of a colored graph where this is not the case.

For $d=3$, we can draw similar pictures. Accompanied with the data of $2$-boxes and $3$-boxes, the following graph 

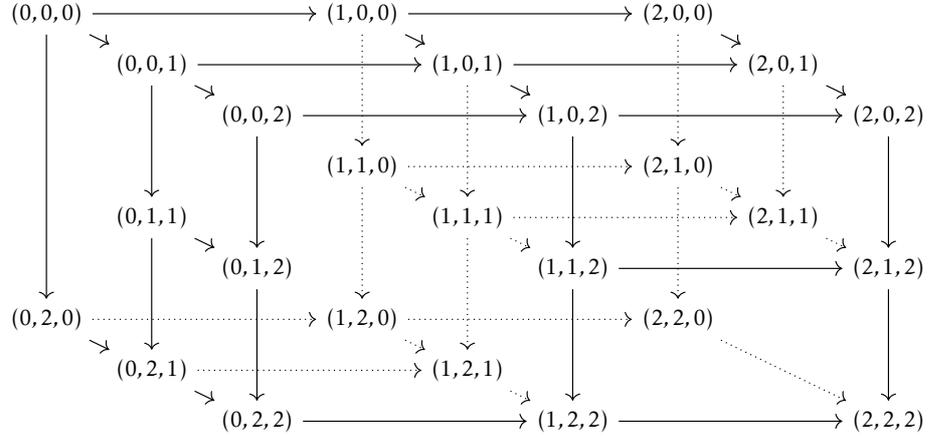
\adjustbox{scale=0.8,center}{
	\begin{tikzcd}[column sep = tiny, row sep = tiny]
		|[alias=000]|(0,0,0)  &&& |[alias=100]|(1,0,0)   &&& |[alias=200]|(2,0,0)\\
		&|[alias=001]| (0, 0, 1)  &&&|[alias=101]| (1,0,1) &&& |[alias=201]|(2,0,1)\\
		& & |[alias=002]|(0,0,2) &&& |[alias=102]|(1,0,2) &&& |[alias=202]|(2,0,2) \\ 
		|[alias=010]|   &&& |[alias=110]|(1,1,0)  &&& |[alias=210]|(2,1,0)\\
		& |[alias=011]|(0, 1, 1)  &&& |[alias=111]|(1,1,1) &&& |[alias=211]|(2,1,1)\\
		& & |[alias=012]|(0,1,2) &&& |[alias=112]|(1,1,2) &&& |[alias=212]|(2,1,2) \\
		|[alias=020]|(0,2,0) &&& |[alias=120]|(1,2,0)  &&& |[alias=220]|(2,2,0)\\
		& |[alias=021]|(0, 2, 1)  &&&|[alias=121]| (1,2,1) &&& |[alias=221]|\\
		& & |[alias=022]|(0,2,2) &&& |[alias=122]|(1,2,2) &&& |[alias=222]|(2,2,2) \\
		\arrow[from=021, to=121, dotted]
		\arrow[from=121, to=122, dotted]
		\arrow[from=111, to =121, dotted]
		\arrow[from=120, to=121,dotted]
		\arrow[from=120, to=220, dotted]
		\arrow[from=110, to=120, dotted]
		\arrow[from=020, to=120, dotted]
		\arrow[from=220, to=222, dotted]
		\arrow[from=210, to=220, dotted]
		\arrow[from=210, to=211, dotted]
		\arrow[from=200, to=210, dotted]
		\arrow[from=110, to=210, dotted]
		\arrow[from=111, to=211, dotted]
		\arrow[from=100,to=110, dotted]
		\arrow[from=110, to=111, dotted]
		\arrow[from=111, to=112, dotted]
		\arrow[from=101, to=111, dotted]
		\arrow[from=201, to=211, dotted]
		\arrow[from=211, to=212, dotted]
		\arrow[from=100, to=200]
		\arrow[from=200, to=201]
		\arrow[from=201, to=202]
		\arrow[from=000, to=100]
		\arrow[from=100, to=101]
		\arrow[from=101, to=102]
		\arrow[from=102, to=112]
		\arrow[from=101, to=201]
		\arrow[from=102, to=202]
		\arrow[from=112, to=212]
		\arrow[from=202, to=212]
		\arrow[from=212, to=222]
		\arrow[from=002, to=102]
		\arrow[from=000, to=001]
		\arrow[from=001, to=101]
		\arrow[from=000, to=020]
		\arrow[from=001, to=011]
		\arrow[from=011, to=012]
		\arrow[from=001, to=002]
		\arrow[from=002,to=012]
		\arrow[from=020,to=021]
		\arrow[from=021, to=022]
		\arrow[from=011, to=021]
		\arrow[from=012, to=022]
		\arrow[from=112, to=122]
		\arrow[from=122, to=222]
		\arrow[from=022, to=122]
	\end{tikzcd}
}
would define a pasting shape $J$. The graph above determines $B^1(J)$, but it now becomes visually intractable to color 2- and 3-boxes in the graph. Instead, we will describe the sets of 2- and 3-boxes in this case. 
For instance, we may take
\begin{align*}
B^3(J) \setminus B^2(J) = N^3 := \{&((0,0,0),(1,2,2)), ((1,0,0),(2,1,1)),((1,0,1), (2,1,2)), \\  
&((1,0,0),(2,1,2)), ((1,1,0), (2,2,2)), ((1,0,0), (2,2,2)), \\ &((0,0,0), (2,2,2))\},
\end{align*} 
for the non-degenerate 3-boxes of $J$.  For the set $$B^2(J) \setminus B^1(J)$$ of non-degenerate 2-boxes of $J$, we can pick all those non-degenerate 2-boxes $(x,y)$ 
with the property that the faces of $(x,y)$ are in the graph above, and there exists a face $(\alpha,\omega) \in B^{3,2}$ of a box in $N^3$ such that $\alpha_a \leq x_a, y_a \leq \omega_a$ for all $a$. 
\end{example}

Note that pasting shapes have a notion of \textit{subshapes}, \textit{unions} and \textit{intersections}:

\begin{definition}
	Let $I$ be a $d$-dimensional pasting shape. A subshape $I' \subset I$ of $I$ is a $d$-dimensional pasting shape $I'$ such that $B^d(I') \subset B^d(I)$.
	
	A subshape $E \subset I$ is called a \textit{$k$-entire subshape of $I$} if there there exists a non-degenerate $k$-box 
	$(\alpha, \omega)$ whose $(k-1)$-faces are in $I$, with the property that any box $(x,y)$ of $I$ is contained in $E$ if and only if $\alpha_a \leq x_a,y_a\leq \omega_a$ for all $a$. 
	In this case, the pair $(\alpha, \omega)$ is called the pair of \textit{bounding box of $E$}. We say that $E$ is \textit{closed} 
	if $(\alpha, \omega) \in I$, in which case $(\alpha, \omega)$ is also contained in $E$. If $E$ is $k$-entire and every $k$-box of $E$ is degenerate, then $E$ is called \textit{open}.

	A $k$-entire subshape $V \subset I$ is called a $k$-\textit{vertebra} of $I$ if for any $k$-entire subshape $E \subset V$, we have $E = V$. 

	If $k=d$, we will drop $d$ from the notation, and consequently call $d$-entire subshapes and $d$-vertebrae of $I$, respectively, entire subshapes and vertebrae.
\end{definition}

\begin{remark}
For a $d$-dimensional pasting shape that belongs to a certain good class of pasting shapes, its vertebrae should be viewed as certain `indecomposable' $d$-dimensional subshapes so that taking their union will always recover the original pasting shape. We will make this precise in \ref{ssection.cpshapes} and \ref{ssection.ptheorem}.
\end{remark}

\begin{example}
	Consider the 2-dimensional pasting shape $I$ of \ref{ex.pshapes-graph} and its subshapes
	
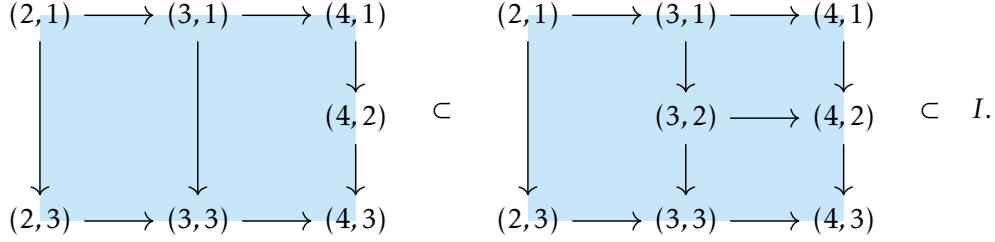
\[
		\begin{tikzcd}[execute at end picture={
			\scoped[on background layer]
			\fill[pamblue, opacity = 0.5] (a.center) -- (b.center) -- (d.center) -- (c.center);
		}]
			|[alias=a]| (2,1)\arrow[dd] \arrow[r, ""'name=f4] & (3,1) \arrow[dd] \arrow[r, ""'name=f5] & |[alias=b]|(4,1)\arrow[d] \\
			 &  & (4,2) \arrow[d]\\
			 |[alias=c]|(2,3) \arrow[r,""name=t4] &(3,3)\arrow[r,""name=t6] & |[alias=d]|(4,3)
		\end{tikzcd}
		\;\;\;\;\subset\;\;\;\;
	\begin{tikzcd}[execute at end picture={
		\scoped[on background layer]
		\fill[pamblue, opacity = 0.5] (a.center) -- (b.center) -- (d.center) -- (c.center);
	}]
		|[alias=a]| (2,1)\arrow[dd] \arrow[r, ""'name=f4] & (3,1) \arrow[d] \arrow[r, ""'name=f5] & |[alias=b]|(4,1)\arrow[d] \\
		 & (3,2)\arrow[d]\arrow[r, ""name=f6] & (4,2) \arrow[d]\\
		 |[alias=c]|(2,3) \arrow[r,""name=t4] &(3,3)\arrow[r,""name=t6] & |[alias=d]|(4,3)
	\end{tikzcd}
	\;\;\;\;\subset\;\;\;\; I.
	\]
	Then the subshape on the left is not entire in $I$, but the second is with bounding box given by $((2,1), (4,3))$.  Note that in this case, $I$ is entire in $I$. 
	The subshape 
	
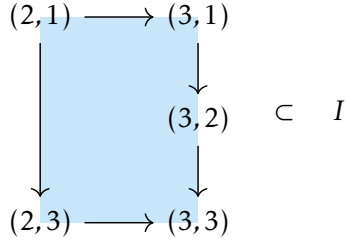
\[
	\begin{tikzcd}[execute at end picture={
		\scoped[on background layer]
		\fill[pamblue, opacity = 0.5] (a.center) -- (b.center) -- (d.center) -- (c.center);
	}]
		|[alias=a]| (2,1)\arrow[dd] \arrow[r, ""'name=f4] & |[alias=b]|(3,1) \arrow[d]  \\
		 & (3,2)\arrow[d] \\
		 |[alias=c]|(2,3) \arrow[r,""name=t4] &|[alias=d]| (3,3)
	\end{tikzcd}
	\;\;\;\;\subset\;\;\;\;\; I
	\]
	is an example of a vertebra. However, the subshape $$(0,1) \rightarrow (2,1)\;\;\;\;\; \subset\;\;\;\;\; I$$  is not an 1-vertebra of $I$ since it is not 
	1-entire: 
	the 1-box $(0,1) \rightarrow (2,1)$ can be obtained by joining the boxes $(0,1) \rightarrow (1,1)$ and $(1,1) \rightarrow (2,1)$ in $I$.
\end{example}

\begin{definition} 
	Let $I$ and $J$ be pasting shapes. Then we define the \textit{union of $I$ and $J$}, $I \cup J$,  to be the smallest pasting shape that contains both $I$ and $J$.
	Moreover, we define the \textit{intersection of $I$ and $J$}, $I \cap J$, to be the largest pasting shape that is contained in $I$ and $J$. Concretely, $I \cap J$ has boxes given by $B^d(I \cap J) = B^d(I) \cap B^d(J)$. 
\end{definition}

We have a few important pasting shapes. For natural numbers $n_1, \dotsc, n_d$, we 
define the $d$-dimensional pasting shape $$\square[n_1, \dotsc, n_d]$$
to be the largest pasting shape whose vertices 
are the $d$-tuples contained in $\{0, \dotsc, n_1\} \times \dotsb \times \{0, \dotsc, n_d\}$, i.e.\ its boxes are given by
$$B^d(\square[n_1, \dotsc, n_d]) = \{(x,y) \in B^{d,d} \mid 0 \leq x_a, y_a \leq n_a \text{ for all $a$}\}.$$ 
With the obvious structure maps, these pasting shapes assemble to a $d$-uple cosimplicial object in $\Shape^d$
$$
\Delta^{\times d} \rightarrow \Shape^d : (n_1, \dotsc, n_d) \mapsto \square[n_1, \dotsc, n_d].
$$
We will show that this cosimplicial object satisfies the following \textit{co-Segal property}: 
\begin{proposition}\label{prop.nerve-cosegal}
	The structure maps of $\square[\bullet, \dotsc, \bullet]$ witness the colimit
	$$
	\colim_{[k_a] \in \mathbb{G}/[n_a]} \square[n_1, \dotsc, n_{a-1}, k_a, n_{a+1} ,\dotsc, n_d] = \square[n_1,\dotsc, n_d].
	$$
	in $\Shape^d$.
\end{proposition}

To prove the above proposition, we will make use of the following description 
of maps out of $\square[n_1,\dotsc, n_d]$:

\begin{proposition}\label{prop.simplices-nerve}
	There is a bijection
	$$
	\textstyle \Shape^d(\square[n_1, \dotsc, n_d], I) \rightarrow \Hom_I(\coprod_{1\leq a\leq d} [n_a], \N),
	$$
	natural in $([n_1], \dotsc, [n_d]) \in \Delta^{\times d}$, where the right-hand side 
	is the subset $$\textstyle\Hom_I(\coprod_{1\leq a\leq d} [n_a], \N)\subset \cat{Poset}(\coprod_{1\leq a\leq d}[n_a],\N) = \prod_{1\leq a\leq d}\cat{Poset}([n_a], \N)$$ of those tuples $(f_1, \dotsc, f_d)$ which have the property that 
	each box $(x,y)$ with $x_a, y_a \in \im(f_a)$ for all $a$, is contained in $I$.
	\end{proposition}
	\begin{proof}
	The key observation is the following. Let  $f : \square[n_1, \dotsc, n_d] \rightarrow I$ be a map of pasting shapes, and define 
	$$f_a(i) := f(0, \dotsc, 0, i, 0, \dotsc, 0)_a.$$
	Then we claim that the image of a vertex $x \in \square[n_1, \dotsc, n_d]$ is given by 
	$
	f(x) = (f_1(x_1), \dotsc, f_d(x_d)).
	$
	Indeed, for any index $1\leq a \leq d$, we may consider the projection $p^a := (0, \dotsc, 0, x_a, 0, \dotsc, 0)$ of $x$.
	Then $(p^a, x)$ is a $(d-1)$-box of $\square[n_1,\dotsc, n_d]$. By the definition of pasting maps, we must then have that 
	$f(p^a)_a = f(x)_a$. By definition, we have $f_a(x_a) = f(p^a)_a$ so that $f(x)_a = f_a(x_a)$.

	The correspondence is now given as follows: we carry a map of pasting shapes $f : \square[n_1, \dotsc, n_d] \rightarrow I$ to the tuple $(f_1, \dotsc, f_d)$ 
	defined above. Note that each $f_a$ is indeed a map of posets $[n_a] \rightarrow \N$.
	There is an inverse to this correspondence: we may carry a tuple $(f_1, \dotsc, f_d) \in \Hom_I(\coprod [n_a], \N)$, 
	to the unique map of pasting shapes $f : \square[n_1, \dotsc, n_d] \rightarrow I$ which is given on vertices by $f(x) := (f_1(x_1), \dotsc, f_d(x_d)).$
\end{proof}

\begin{proof}[Proof of \ref{prop.nerve-cosegal}]
	Let $I$ be a $d$-dimensional pasting shape, and let us write $H_k$ for the set of tuples  
	$
	 \Hom_I([n_1] \sqcup \dotsc \sqcup [n_{a -1 }] \sqcup [k] \sqcup [n_{a+1}] \sqcup \dotsc \sqcup [n_d], \N)$. In light of \ref{prop.simplices-nerve}, we have to verify that the canonical map 
	$$
	\textstyle H_{n_a} \rightarrow H_{1} \times_{H_0} \dotsc \times_{H_0} H_1
	$$ is a bijection. Note that $[n_a]$ can be written as the iterated pushout
	$
	[1] \cup_{[0]} \dotsc \cup_{[0]} [1]
	$
	of posets. Hence, the canonical map can be identified with the inclusion 
	$$
	\textstyle H_{n_a} \rightarrow H'_{n_a},
	$$
	where the subset $H'_{n_a} \subset \Hom(\coprod_{1\leq a\leq d} [n_a], \N)$ contains those tuples $(f_1, \dotsc, f_d)$ which have the following property: if $(x,y)$ is 
	a box so that for each index $b$, there exist indices $i_b \leq j_b$ so that 
	$$x_b = f_b(i_b),\quad  y_b = f_b(j_b),$$ and 
	$j_b - i_b \leq 1$ if $b=a$, then $(x,y)$ is contained in $I$. One now readily verifies 
	that $H'_{n_a} = H_{n_a}$ by making use of the fact that pasting shapes are closed under joins.
\end{proof}

This observation now leads to the following construction:

\begin{construction}
	Given a $d$-dimensional pasting shape $I$, we define a $d$-uple simplicial set $[I]$, the \textit{nerve of $I$}, by setting 
	$$[I]_{n_1, \dotsc, n_d} := \Shape^d(\square[n_1, \dotsc, n_d], I).$$
	On account of \ref{prop.nerve-cosegal}, $[I]$ is a $d$-uple Segal set. This is a canonically functorial construction, so that 
	we obtain a functor
	$$[-] : \Shape^d \rightarrow \Cat^d(\cat{Set})$$
	that carries pasting shapes to their nerves. 
	Note that the simplices of the nerve may be described using \ref{prop.simplices-nerve}. In particular, 
	one can use this description to readily verify that $[-]$ is fully faithful.
\end{construction}

\begin{corollary}\label{cor.cosimp-objs}
	The $d$-uple cosimplicial object $[\square[\bullet, \dotsc, \bullet]]$ concides with $\Delta[\bullet, \dotsc, \bullet]$.
\end{corollary}
\begin{proof}
	In light of \ref{prop.simplices-nerve}, we see that $$[\square[m_1, \dotsc, m_d]]_{n_1, \dotsc, n_d} \cong \prod_{1 \leq a \leq d}\Delta([n_a], [m_a]),$$ 
	natural in all variables.
\end{proof}

\subsection{Truncations of pasting shapes and grids}
In order to define a good class of pasting shapes, we need to introduce the building blocks of 
these shapes: so-called \textit{grids}. These are pasting shapes that are in some sense shaped by the standard grids $\square[n_1,\dotsc,n_d]$.
Let us start with an illustrative example. Consider the obvious injective map of pasting shapes
\[
	\square[2,1] \;\;\;\;\;\;\;\;\;\;\; \xrightarrow{\makebox[45pt]{}} \;\;\;\;\;\;\;\;\;\;\;
	\begin{tikzcd}[execute at end picture={
		\scoped[on background layer]
		\fill[pamblue, opacity = 0.5] (a.center) -- (b.center) -- (d.center) -- (c.center);
	}]
		|[alias=a]|(0,0) \arrow[r]\arrow[d] & (1,0)\arrow[dd] \arrow[r] & |[alias=b]|(2,0)\arrow[dd] \\
		(0,1) \arrow[d]\\
		|[alias=c]|(0,2) \arrow[r] & (1,2) \arrow[r] &|[alias=d]|(2,2).
	\end{tikzcd}
\]
Note that this map does not create any new 2-cells, but factors one of the boundary 1-cells. The codomain is in this sense shaped by the domain (modulo boxes of dimension $\leq 1$, they agree), and would be an example of a grid. We want to make this precise and capture this notion.
First, we observe that $d$-uple simplicial sets support a notion of dimension:

\begin{definition}
	Suppose that $S$ is a $d$-uple simplicial set. Then a $d$-uple simplex $\sigma : \Delta[n_1, \dotsc, n_d] \rightarrow S$ 
	is called \textit{of dimension $\leq k$} if there exists a factorization 
	$$
	\Delta[n_1, \dotsc, n_d] \rightarrow \Delta[m_1, \dotsc, m_d] \rightarrow S
	$$
	of $\sigma$
	such that there are at most $k$ indices $i$ such that $m_i \neq 0$. 
	Moreover, we say that $\sigma$ is \textit{of dimension $k$} if it is of dimension $\leq k$ but not of dimension $\leq (k-1)$.
	The $d$-uple simplicial subset of $S$ containing those $d$-uple simplices of dimension $\leq k$ is called the \textit{$k$-truncation of $S$} and denoted by $$\dtr_{\leq k} S \subset S.$$ 
\end{definition}

Whenever $S$ is the nerve of a pasting shape, the truncation $\dtr_{\leq k}S$ can be described as follows.

\begin{definition}
Let $I$ be a $d$-dimensional pasting shape. Then for $0 \leq k \leq d$, we define its \textit{$k$-truncation} $\dtr_{\leq k}I \subset I$ to be the pasting shape with boxes 
$$
B^i(\dtr_{\leq k}I) = \begin{cases}
	B^i(I) & \text{if $i \leq k$}, \\
	B^k(I) & \text{if $i \geq k$}.
\end{cases}
$$
We say that $I$ is \textit{k-truncated} if $\dtr_{\leq k}I = I$, i.e., whenever every $l$-box of $I$ is degenerate for $l > k$.
\end{definition}

\begin{example}
	The following is an important class of examples of $k$-truncated $d$-dimensional pasting shapes. For $d-k$ different indices $1 \leq a_1, \dotsc, a_{d-k} \leq d$
	and integers $c_1, \dotsc, c_{d-k}$, we can consider the $k$-dimensional hyperplane $H$ whose boxes are given by
	$$
	B^d(H) = \{(x,y) \in B^{d,d} \mid x_{a_i} = y_{a_i} = c_i \text{ for all $i = 1,\dotsc, d-k$}\}.
	$$
	We will denote the set of $k$-dimensional hyperplanes by $\mathscr{H}_k$.
\end{example}

The following can readily be deduced from the two definitions of truncations that we have seen above:

\begin{proposition}\label{prop.desc-trunc}
Suppose that $I$ is a $d$-dimensional pasting shape, and let $0 \leq k \leq d$. Then 
$$
\dtr_{\leq k}I = \bigcup_{H \in \mathscr{H}_k} I \cap H,
$$
and 
$$
\dtr_{\leq k}[I] = [\dtr_{\leq k} I]= \bigcup_{H \in \mathscr{H}_k} [I\cap H].
$$
\end{proposition}

\begin{remark}\label{rem.hyperplane-embedding}
	Suppose that we have a tuple $h \in \N^{\times k}$ of indices $1 \leq h_1 <  \dotsb < h_{k} \leq d$. Then we obtain a 
	map of posets $m_{h} : \N^{\times k} \rightarrow \N^{\times d}$ so that $m_h(x)_{j} = x_i$ if $j= h_i$ and $m_h(x)_{j} = 0$ otherwise. 
	This map gives rise to a fully faithful functor
	$$
	i_h : \Shape^k \rightarrow \Shape^d 
	$$
	that carries a $k$-dimensional pasting shape $I$ to the $k$-truncated $d$-dimensional pasting shape whose boxes are given by the image 
	of the boxes of $I$ under $m_h \times m_h$. The essential image of $i_h$ is given by those $d$-dimensional pasting shapes that are contained in some hyperplane $H \in \mathcal{H}_k$ 
	whose associated indices are different from the $h_i$'s. 
	Hence, will implicitly view any $d$-dimensional pasting shape that is contained in some hyperplane $H \in \mathcal{H}_k$ as a $k$-dimensional pasting shape. 
	Note that the embedding $i_h$ is compatible with the nerve functors, in the sense that we have a commutative square 
	\[
		\begin{tikzcd}
			\Shape^k \arrow[r, "i_h"]\arrow[d,"{[-]}"'] & \Shape^d\arrow[d,"{[-]}"] \\
			\Cat^k(\cat{Set}) \arrow[r, "p_h^*"] & \Cat^d(\cat{Set}),
		\end{tikzcd}
	\]
	where $p_h^*$ is as constructed in \ref{ssection.cat-objs}. 
\end{remark}

Using the notion of truncations, we can give a description of what it means for a $d$-dimensional pasting shape $I$ to shape 
a $d$-dimensional pasting shape $J$ modulo $(d-1)$-boxes.

\begin{proposition}
	Let $f : I \rightarrow J$ be an injective map of $d$-dimensional pasting shapes. Then the following are equivalent: 
\begin{enumerate}[noitemsep]
		\item every non-degenerate $d$-box of $J$ is the image of a non-degenerate $d$-box in $I$,
		\item $J$ can be written as $$J = f(I) \cup \dtr_{\leq d-1}J,$$
		\item every $d$-dimensional simplex of $[J]$ lies in $f([I])$,
		\item the nerve of $J$ can be written as $$[J] = f([I]) \cup \dtr_{\leq d-1}[J].$$
\end{enumerate}
\end{proposition}
\begin{proof}
We only prove that (1) implies (3). The remaining implications are readily verified. Suppose that we have a $d$-dimensional simplex of $[J]$ corresponding to a map 
$$
g : \square[n_1, \dotsc, n_d] \rightarrow J,
$$
then we must show that $g$ factors through $I$. Since $f$ is injective, it induces an isomorphism $f : I \rightarrow f(I) \subset J$ 
to the image $f(I)$ of $I$. Hence, we must show that $g$ has image in $f(I)$. 
We may assume that each map $g_a : [n_a] \rightarrow \N$ is injective and $n_a \neq 0$. Now, any box $(x,y)$ of $\square[n_1, \dotsc, n_d]$ 
lies in the image of an injective map $i : \square[1, \dotsc, 1] \rightarrow \square[n_1, \dotsc, n_d]$. Thus it suffices to show that $gi$ has image in $f(I$). 
But $gi$ classifies a non-degenerate $d$-box of $J$, 
hence has image in $f(I)$ by assumption. 
\end{proof}
	
\begin{definition}
	We call a map of $d$-dimensional pasting shapes $I \rightarrow J$ \textit{$d$-shaping} if it is injective and satisfies any of the equivalent conditions above.
\end{definition}

With this terminology now in place, we can introduce the following notion:

\begin{definition}
A $d$-dimensional pasting shape $A$ is called a \textit{grid}, more specifically, an \textit{open grid}, respectively, a \textit{closed grid}, if 
there exists a $d$-shaping map
$$f : \dtr_{\leq d-1}\square[n_1, \dotsc, n_d] \rightarrow A, \text{ resp. } f : \square[n_1, \dotsc, n_d] \rightarrow A,$$
with $n_1, \dotsc, n_d \neq 0$, having the property that any $(d-1)$-box $(x,y)$ in $A$ 
satisfies $0 \leq x_a, y_a \leq f_a(n_a)$ for all indices $a$, and there exists an index $b$ such that $x_b = y_b \in \im(f_b)$. 
\end{definition}

\begin{example}\label{ex.grids} The 2-dimensional pasting shape $$
	\begin{tikzcd}[column sep = 15, row sep = 15, execute at end picture={
		\scoped[on background layer]
		\fill[pamblue, opacity = 0.5] (a.center) -- (b.center) -- (d.center) -- (c.center);
	}]
		|[alias=a]|(0,0) \arrow[r]\arrow[d] & (1,0)\arrow[dd] \arrow[r] & (2,0)\arrow[dd] \arrow[rr] & & |[alias=b]|(4,0)\arrow[dd] \\
		(0,1) \arrow[d]\\
		|[alias=c]|(0,2) \arrow[r] & (1,2) \arrow[r] &(2,2) \arrow[r] & (3,2) \arrow[r] & |[alias=d]| (4,2)
	\end{tikzcd}
	$$
	is an example of a closed grid, and the 2-dimensional pasting shape 
	\[ 
		\begin{tikzcd}[column sep = 15, row sep = 15]
			|[alias=a]|(0,0) \arrow[r]\arrow[d] & (1,0) \arrow[r] & |[alias=b]|(2,0)\arrow[d] \arrow[r] & (3,0) \arrow[d] \arrow[rr] && (5,0)\arrow[d] \\
			(0,1) \arrow[d]\arrow[rr] & & (2,1)\arrow[d] \arrow[r] & (3,1) \arrow[d] \arrow[r] & (4,1) \arrow[r] & (5,1)\arrow[d]\\
			|[alias=c]|(0,2) \arrow[r] & (1,2) \arrow[r] &|[alias=d]|(2,2) \arrow[r] & (3,2) \arrow[r] & (4,2) \arrow[r] & (5,2)
		\end{tikzcd}	
	\]
	is an open grid.
	But note that the following 2-dimensional pasting shape \textit{fails} to be a grid:
	\[
		\begin{tikzcd}[column sep = 15 pt, row sep = 15]
			(0,0) \arrow[r]\arrow[dd] & (1,0) \arrow[d] \arrow[r] & (2,0) \arrow[dd] \\
			& (1,1) & \\
			(0,2)\arrow[rr] && (2,2).
		\end{tikzcd}
	\]
\end{example}

\begin{proposition}
	Let $A$ be a grid witnessed by a map $f$. Then this map $f$ is unique: every other map that witnesses $A$ to be a grid coincides with $f$.
\end{proposition}
\begin{proof}
	Note that $A$ cannot be open and closed simultaneously. We restrict to the case that $A$ is open as the closed case is handled analogously. Let
	$$
	g : \dtr_{\leq d-1}\square[m_1,\dotsc, m_d] \rightarrow A
	$$
	be another map that witnesses $A$ to be a grid. Similarly to the analysis of \ref{prop.simplices-nerve}, we can identify $f$ and $g$ respectively 
	with tuples $(f_1, \dotsc, f_d)$, $(g_1, \dotsc, g_d)$ such that each $(d-1)$-box $(x,y)$ with $x_a, y_a \in \im(f_a)$, resp.\ $x_a,y_a \in \im(g_a)$, is in $A$.
	Since $f$ and $g$ are injective, it suffices to show that $\im(f_b) = \im(g_b)$ for each index $1 \leq b \leq d$. Let $x_b \in \im(f_b)$. Then for $a \neq b$, we may find integers $x_a < y_a$ in $\im(f_a)$ since $n_a \neq 0$.
	Setting $y_b := x_b$, we obtain a non-degenerate $(d-1)$-box $(x,y)$ in $A$. Since $g$ witnesses $A$ to be a grid, we must then have that $x_b = y_b \in \im(g_b)$. 
	Hence $\im(f_b) \subset \im(g_b)$. Interchanging the roles of $f$ and $g$, we also obtain the reverse inclusion.
\end{proof}

The above observation shows that the following definition is well-defined.

\begin{definition}
Let $A$ be a grid witnessed by a map $f$.
We call the $d$-tuple of points $$((f_1(0), f_1(n_1)), \dotsc, (f_d(0), f_d(n_d)))$$ the \textit{corners} of $A$.
We define the \textit{boundary} $$\partial A \subset A$$ to be the subshape of $A$ 
that consists of the boxes $(x,y)$ in $A$ such that $x_a = y_a = f_a(0), f_a(n_a)$ for some index $a$. 
If $n_1 = \dotsb = n_d = 1$, then $A$ is called a \textit{cell}. 
\end{definition}

\begin{example} The boundary of the closed grid exhibited in \ref{ex.grids} is given by
	\[
		 \begin{tikzcd}[column sep = 15, row sep = 15]
			|[alias=a]|(0,0) \arrow[r]\arrow[d] & (1,0) \arrow[r] & (2,0) \arrow[rr] & & |[alias=b]|(4,0)\arrow[dd] \\
			(0,1) \arrow[d]\\
			|[alias=c]|(0,2) \arrow[r] & (1,2) \arrow[r] &(2,2) \arrow[r] & (3,2) \arrow[r] & |[alias=d]| (4,2).
		\end{tikzcd}
	\]
\end{example}

\begin{example}\label{ex.trivial-grid}
	The pasting shape $\square[n_1, \dotsc, n_d]$ with $n_1, \dotsc, n_d \neq 0$ is of course a grid. Note that its vertebrae 
	are given by the cells 
	$$
	V_{(j_1, \dotsc, j_d)}, \quad 0 \leq j_a \leq n_a - 1,
	$$
	which is the largest $d$-dimensional pasting shape whose vertices $x$ satisfy $x_a = j_a, j_a + 1$ for all $a$.

	From the definition of grids, it may be readily deduced that the vertebrae of a grid $A$ witnessed by a map 
	$$f : \dtr_{\leq d-1}\square[n_1, \dotsc, n_d] \rightarrow A, \quad \text{or}, \quad f : \square[n_1, \dotsc, n_d] \rightarrow A,$$
	are given by cells whose corners are given by the corners of $f(V_j)$, for some tuple $j$ as above.
\end{example}

\subsection{Admittable and composable pasting shapes}\label{ssection.cpshapes}

Suppose that $I$ is an entire $d$-dimensional pasting shape. Then we would like to show that $I$ can be obtained by pasting its vertebrae together. Precisely, 
we would like to show that $I$ is the union of its vertebrae $$I = \bigcup_{V \subset I \text{ vertebra}} V,$$
and more importantly, that this colimit is preserved by the nerve functor $[-] : \Shape^d \rightarrow \Cat^d(\S)$. However, 
this is not true in general. 
For instance, if $d=2$, then we can consider the famous \textit{pinwheel} introduced by Dawson and Par\'e \cite{DawsonPare}. This is the 2-uple Segal set associated with the pasting shape
\[
	PW := 
	\begin{tikzcd}[execute at end picture={
		\scoped[on background layer]
		\fill[pamblue, opacity = 0.5] (a.center) -- (b.center) -- (d.center) -- (c.center);
	}]
		|[alias=a]|(0,0) \arrow[d]\arrow[rr] && (2,0) \arrow[r]\arrow[d] & |[alias=b]|(3,0)\arrow[dd] \\
		(0,1)\arrow[dd] \arrow[r] & (1,1) \arrow[r] \arrow[d]& (2,1) \arrow[d] & \\
		& (1,2)\arrow[d]\arrow[r] & (2,2) \arrow[r] & (3,2) \arrow[d]\\ 
		|[alias=c]|(0,3)\arrow[r] & (1,3)\arrow[rr] && |[alias=d]|(3,3).
	\end{tikzcd}
\]
One can show that mapping out of $[PW]$ into a 2-uple Segal space, is not equivalent to mapping out of the nerves of its vertebrae in a compatible fashion. 
Already at the level of pasting shapes, the pinwheel cannot be written as the union of its vertebrae. We will demonstrate this 
in \ref{ex.no-baby-pasting-pinwheel} and \ref{ex.no-pasting-pinwheel}.

Hence, we will need to single out a class of pasting shapes for which such a pasting result holds. The initial objective of this subsection 
is to introduce a candidate: the class of \textit{composable} pasting shapes.  First, we consider an intermediate, auxiliary class of \textit{admittable} pasting shapes.
Intuitively, these are the pasting shapes that can be built out of grids.
After defining these classes of pasting shapes, will then continue to show that the admittable and composable pasting shapes are considerably better behaved than arbitrary pasting shapes. In particular, we will show that any composable shape can be written as the union of its vertebrae. This can be considered 
as an early version of the \textit{pasting theorem} introduced in \ref{ssection.ptheorem}, and similar ideas will be used as input to prove the full strength of the pasting theorem.

\begin{definition}\label{def.admittable-shapes}
	A $d$-dimensional pasting shape $I$ is called \textit{admittable} if there 
	exists a filtration 
	$$
	I_0 \subset I_1 \dotsb \subset I_n = I,
	$$
	where $I_0$ is a grid and $I_k = I_{k-1} \cup A_k$ for a grid $A_k \subset I$ such that $I_{k-1} \cap A_k = \partial A_k$ is an open vertebra of $I_{k-1}$.
	The \textit{corners of $I$} are given by the corners of $I_0$ (it is readily verified that this is independent of the filtration). Moreover, we define 
	the \textit{boundary of $I$}  $$\partial I \subset I$$ to be the smallest subshape of $I$ that contains 
	all boxes $(x,y)$ such that  $x_a = y_a$ is equal to a coordinate of the $a$'th corner for some index $a$, or equivalently $\partial I = \partial I_0$.
\end{definition}

\begin{example}\label{ex.admittable}
	The filtration of 2-dimensional pasting shapes
	
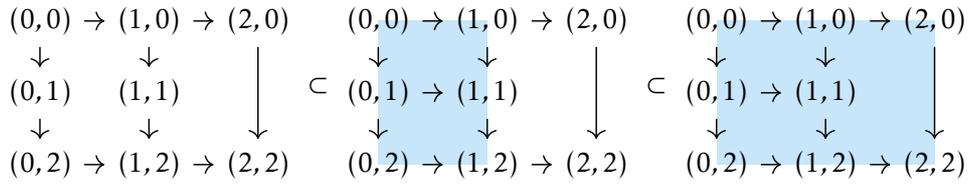
\[
		\begin{tikzcd}[column sep = 8, row sep = 8]
			|[alias=a]| (0,0)\arrow[d] \arrow[r, ""'name=f4] & (1,0) \arrow[d] \arrow[r, ""'name=f5] & |[alias=b]|(2,0)\arrow[dd] \\
			(0,1)\arrow[d]  & (1,1)\arrow[d] & \\
			 |[alias=c]|(0,2) \arrow[r,""name=t4] &(1,2)\arrow[r,""name=t6] & |[alias=d]|(2,2)
		\end{tikzcd}
		\subset
		\begin{tikzcd}[column sep = 8, row sep = 8, execute at end picture={
			\scoped[on background layer]
			\fill[pamblue, opacity = 0.5] (a.center) -- (b.center) -- (d.center) -- (c.center);
		}]
			|[alias=a]| (0,0)\arrow[d] \arrow[r, ""'name=f4] & |[alias=b]|(1,0) \arrow[d] \arrow[r, ""'name=f5] & (2,0)\arrow[dd] \\
			(0,1)\arrow[d] \arrow[r] & (1,1)\arrow[d] & \\
			 |[alias=c]|(0,2) \arrow[r,""name=t4] &|[alias=d]|(1,2)\arrow[r,""name=t6] & (2,2)
		\end{tikzcd}
		\subset
		\begin{tikzcd}[column sep = 8, row sep = 8, execute at end picture={
			\scoped[on background layer]
			\fill[pamblue, opacity = 0.5] (a.center) -- (b.center) -- (d.center) -- (c.center);
		}]
			|[alias=a]| (0,0)\arrow[d] \arrow[r, ""'name=f4] & (1,0) \arrow[d] \arrow[r, ""'name=f5] & |[alias=b]|(2,0)\arrow[dd] \\
			(0,1)\arrow[d] \arrow[r] & (1,1)\arrow[d] & \\
			 |[alias=c]|(0,2) \arrow[r,""name=t4] &(1,2)\arrow[r,""name=t6] & |[alias=d]|(2,2)
		\end{tikzcd}
\]
witnesses the rightmost pasting shape to be admittable.
\end{example}

\begin{definition}\label{def.composable-shapes}
A $d$-dimensional pasting shape $I$ is called \textit{composable} if it meets the following conditions: 
\begin{enumerate}[noitemsep]
	\item $I$ is closed admittable (i.e.\ admittable and closed $d$-entire),
	\item any closed $k$-entire subshape of $I$, with $k < d$, is admittable when viewed as a $k$-dimensional pasting shape (see \ref{rem.hyperplane-embedding}).
\end{enumerate}
\end{definition}

\begin{example} Here are some examples:
\begin{itemize}[noitemsep ]
	\item The 1-dimensional admittable pasting shapes correspond to those subposets of $\N$ which are disjoint unions of intervals. The composable ones correspond to the intervals.
	\item An example of a 2-dimensional composable pasting shape is the pasting shape 
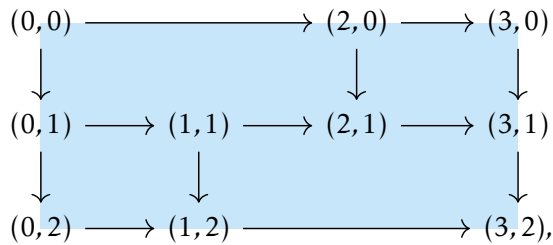
\[
\begin{tikzcd}[execute at end picture={
	\scoped[on background layer]
	\fill[pamblue, opacity = 0.5] (a.center) -- (b.center) -- (d.center) -- (c.center);
}]
	|[alias=a]|(0,0) \arrow[rr]\arrow[d] && (2,0)\arrow[d] \arrow[r] & |[alias=b]|(3,0)\arrow[d] \\
	(0,1) \arrow[r]\arrow[d] & (1,1) \arrow[d]\arrow[r] & (2,1) \arrow[r] & (3,1)\arrow[d] \\
	|[alias=c]|(0,2) \arrow[r] & (1,2) \arrow[rr] && |[alias=d]|(3,2),
\end{tikzcd}
\]
which classifies one of the relevant composites of the triangle identity in a \textit{2-fold Segal space} (see \ref{section.outlook}).
\item A filtration as in \ref{def.admittable-shapes} for the pinwheel $PW$ cannot exist 
since the only subgrids of the pinwheel are six distinct cells. Hence, the pinwheel is not admittable and, in particular, 
not composable.
\item The 2-dimensional pasting shape $I$ and the 3-dimensional pasting shape $J$ of \ref{ex.pshapes-graph} are composable.
\item All pasting shapes in the filtration of \ref{ex.admittable} are admittable, but only the rightmost shape is closed and composable.
\end{itemize}
\end{example}

We would now like to study these admittable and composable pasting shapes. In particular, we will show at the end of this subsection the following theorem which may be considered 
as a stepping stone to \ref{thm.pasting-theorem}, the general pasting theorem:

\begin{theorem}\label{thm.baby-pasting-theorem}
	Let $I$ be a composable pasting shape. Then each vertebra of $I$ is composable, and $I$ can be written as the union of its vertebrae.
\end{theorem}

In order to prove structural properties of admittable pasting shapes (it is also used in the proof of the general pasting theorem), the following notion of \textit{division pairs} turns out to be useful. It is precisely the datum 
of a stage in a filtration that appears in \ref{def.admittable-shapes}:

\begin{definition}
	A \textit{division pair} for an admittable pasting shape $I$ is a tuple $(K,J)$ of admittable subshapes of $I$ such that $K \cap J = \partial J$ is an open vertebra of $K$, and $I = K \cup J$.
\end{definition}

\begin{example}
	An example of a division pair is the pair
	\[	
		\left( \begin{tikzcd}[column sep = 8, row sep = 8,execute at end picture={
			\scoped[on background layer]
			\fill[pamblue, opacity = 0.5] (a.center) -- (b.center) -- (d.center) -- (c.center);
		}]
			 (0,0)\arrow[d] \arrow[r, ""'name=f4] &|[alias=a]| (1,0) \arrow[d] \arrow[r, ""'name=f5] & |[alias=b]|(2,0)\arrow[dd] \\
			(0,1)\arrow[d]  & (1,1)\arrow[d] & \\
			 (0,2) \arrow[r,""name=t4] &|[alias=c]|(1,2)\arrow[r,""name=t6] & |[alias=d]|(2,2),
		\end{tikzcd}\;\;
		\begin{tikzcd}[column sep = 8, row sep = 8, execute at end picture={
			\scoped[on background layer]
			\fill[pamblue, opacity = 0.5] (a.center) -- (b.center) -- (d.center) -- (c.center);
		}]
			|[alias=a]| (0,0)\arrow[d] \arrow[r, ""'name=f4] & |[alias=b]|(1,0) \\
			(0,1)\arrow[d] \arrow[r] & (1,1)\arrow[d]  \\
			 |[alias=c]|(0,2) \arrow[r,""name=t4] &|[alias=d]|(1,2)
		\end{tikzcd}
		\right)
		\quad \text{for}\quad
		\begin{tikzcd}[column sep = 8, row sep = 8, execute at end picture={
			\scoped[on background layer]
			\fill[pamblue, opacity = 0.5] (a.center) -- (b.center) -- (d.center) -- (c.center);
		}]
			|[alias=a]| (0,0)\arrow[d] \arrow[r, ""'name=f4] & (1,0) \arrow[d] \arrow[r, ""'name=f5] & |[alias=b]|(2,0)\arrow[dd] \\
			(0,1)\arrow[d] \arrow[r] & (1,1)\arrow[d] & \\
			 |[alias=c]|(0,2) \arrow[r,""name=t4] &(1,2)\arrow[r,""name=t6] &|[alias=d]| (2,2).
		\end{tikzcd}
\]
\end{example}

\begin{proposition}\label{prop.div-pair-boxes}
	Let $I$ be an admittable pasting shape. Then the following is true: 
	\begin{enumerate} 
		\item For any open vertebra $V$ of $I$ with corners $((\alpha_1, \omega_1), \dotsc, (\alpha_d, \omega_d))$, every box $(x,y)$ in $I$ 
		has $y_a \leq \alpha_a$ or $x_a \geq \omega_a$ for some index $a$.
		\item If $(K,J)$ is a division pair for an admittable pasting shape $I$, where $J$ has corners given by $((\alpha_1, \omega_1), \dotsc, (\alpha_d,\omega_d))$, 
		then $J$ is entire, and for any box $(x,y)$ in $I$ we have that 
	$(x,y)$ is in $K$ if and only if there exists an index $a$ such that $y_a \leq \alpha_a$ or $x_a \geq \omega_a$. 
	\end{enumerate}
\end{proposition}
\begin{proof}
	Suppose that $(K,J)$ is a division pair for an admittable pasting shape $I$ such that (1) holds for $K$ and $J$.
	Denote the corners of $J$ by $((\alpha_1, \omega_1), \dotsc, (\alpha_d, \omega_d))$.
	We inductively define a subset $N^k$ of non-degenerate $k$-boxes of $I$ as follows. We set $N^0 := B^0(K) \cup B^0(J)$. If $k \geq 1$ and $N^{k-1}$ 
	is defined, then we define $N^k$ to be the set of non-degenerate $k$-boxes $(x,y)$ in $I$ such that the faces 
	of $(x,y)$ are in $N^{k-1}$ and one of the following properties is met:
	\begin{enumerate}[label=(\roman*)]
		\item $(x,y) \in B^k(K) \cup B^k(J)$,
		\item for all indices $a$, we have $x_a < \omega_a$ and $y_a > \alpha_a$, and there exists an index $b$ so that $x_b < \alpha_b$ or 
		$y_b > \omega_b$.
	\end{enumerate}
	Here, condition (ii) will capture boxes that are obtained by joining boxes in $K$ with boxes in $J$ which do not lie in $K \cap J = \partial J$. Note that a non-degenerate $k$-box 
	cannot have properties (i) and (ii) simultaneously.

	One now readily checks that the boxes in the subset
	$$
	\bigcup_{k=0}^d N^k \subset B^d(I)
	$$
	are closed under faces and joins (by induction on $k$), and hence, constitute a subshape $I' \subset I$.
	By construction, $K$ and $J$ are contained in $I'$. Since $I = K \cup J$, we must have that $I=I'$. Thus the non-degenerate $k$-boxes 
	of $I$ are given by $N^k$. From this description and the fact that (1) holds for $K$, it can be deduced 
	that (2) holds. Moreover, using that (1) also holds for $J$, we can now deduce that (1) holds for $I$ as well. 
	Hence, we may reduce the proof of the statement to showing that (1) holds any grid, which is clear.
\end{proof}

In the demonstration of \ref{prop.div-pair-boxes} above, we have used division pairs to induct our way up to each admittable pasting shape. Another convenient 
way of giving similarly flavored inductive proofs is using the notion of \textit{heights}:

\begin{definition}
A \textit{decomposition} of an admittable pasting shape $I$ is a pair $(A, \{J_i\})$, where $A \subset I$ is an open subgrid of $I$, $\{J_i\}$ is a collection 
of admittable subshapes of $I$, such the intersections $A \cap J_i = \partial J_i$ are all  mutually distinct open vertebrae of $A$, and
$I$ can be written as 
$$
I = A \cup \bigcup J_i.
$$

We say that an admittable pasting shape $I$ has  \textit{height} 0 if it is a cell. Inductively, we say that a pasting shape has \textit{height $h$} if 
there exists a decomposition $(A, \{J_i\})$ of $I$ where each $J_i$ height $h-1$.
\end{definition}

\begin{example}\label{ex.decomp} The admittable pasting shape
\[
I := \begin{tikzcd}[column sep = 8, row sep = 8, execute at end picture={
	\scoped[on background layer]
	\fill[pamblue, opacity = 0.5] (a.center) -- (b.center) -- (d.center) -- (c.center);
}]
	|[alias=a]| (0,0)\arrow[d] \arrow[r] & (1,0) \arrow[d] \arrow[r] & (2,0)\arrow[d]\arrow[r] & (3,0) \arrow[r] &|[alias=b]| (4,0)\arrow[d]\\
	(0,1) \arrow[d] & (1,1)\arrow[r]\arrow[d] & (2,1)\arrow[d]\arrow[r] & (3,1)\arrow[d]\arrow[r] & (4,1)\arrow[d] \\
	 (0,2) \arrow[r]\arrow[d] &(1,2)\arrow[d]\arrow[r] & (2,2)\arrow[r]\arrow[d] & (3,2)\arrow[r]\arrow[d] & (4,2)\arrow[dd] \\
	 (0,3)\arrow[r]\arrow[d]  &(1,3)\arrow[r] & (2,3)\arrow[d]\arrow[r] & (3,3)\arrow[d] &  \\
	 |[alias=c]|(0,4) \arrow[rr] & & (2,4)\arrow[r] & (3,4)\arrow[r] &|[alias=d]| (4,4)
\end{tikzcd}
\]
has a decomposition with a corresponding open subgrid 
\[
A := \begin{tikzcd}[column sep = 8, row sep = 8]
	|[alias=a]| (0,0)\arrow[d] \arrow[r] & (1,0) \arrow[r] & (2,0)\arrow[d]\arrow[r] & (3,0) \arrow[r] &|[alias=b]| (4,0)\arrow[d]\\
	 (0,1)\arrow[d] &  & (2,1)\arrow[d] & & (4,1)\arrow[d] \\
	 (0,2) \arrow[r]\arrow[d] &(1,2)\arrow[r] & (2,2)\arrow[r]\arrow[d] & (3,2)\arrow[r] & (4,2)\arrow[dd] \\
	 (0,3)\arrow[d]  & & (2,3)\arrow[d] & &  \\
	 |[alias=c]|(0,4) \arrow[rr] & & (2,4)\arrow[r] & (3,4)\arrow[r] &|[alias=d]| (4,4).
\end{tikzcd}
\]
\end{example}

The proof of the following proposition shows how one can obtain a decomposition for an admittable pasting shape:

\begin{proposition}
Any admittable pasting shape admits a decomposition and has a height.
\end{proposition}
\begin{proof}
Suppose that $I$ is an admittable pasting shape with a filtration as in \ref{def.admittable-shapes}. 
Then for $1 \leq k \leq n$, we consider the pair
$$
D_k = (\dtr_{\leq d-1}I_0, S_k),
$$
where $S_k$ is the set of entire subshapes of $I_k$ whose bounding boxes are given by the corners of vertebrae of $I_0$. We 
show by induction that $D_k$ is a decomposition for $I_k$. For $k = 0$, this is clear.  
Suppose that $D_{k-1}$ is a decomposition for $I_{k-1}$. Now, $I_k = I_{k-1} \cup A_{k}$, where $A_k$ is a grid and $A_k \cap I_{k-1} = \partial A_k$ is a vertebra of $I_{k-1}$. \ref{prop.div-pair-boxes} implies 
that $\partial A_k \subset E$ of some subshape $E \in S_{k-1}$. Let $E'$ be the entire subshape in $S_k$ with the same bounding box. 
Then $E' = (I_{k-1} \cup A_k) \cap E' = E \cup A_k$ and $E \cap A_k = \partial A_k$ must be a vertebra of $E$. Thus $E'$ is again admittable. Moreover, we see that $S_k = S_{k-1} \setminus \{E\} \cup \{E'\}$. 
All in all, this implies that $D_k$ is a decomposition for $I_k$. It is readily verified that $I$ has a height.
\end{proof}

\begin{remark}
	Note that the height of an admittable pasting shape $I$ is not unique. If $I$ has height $h$ then $I$ has all heights greater than $h$ as well.
\end{remark}

\begin{example}
	The 2-dimensional admittable pasting shape $I$ of \ref{ex.pshapes-graph} has (minimal) height $3$.
\end{example}

The following \textit{factorization property} is an important property of admittable pasting shapes:

\begin{lemma}\label{lemma.decomp-fact-prop}
	Suppose that $I$ is an admittable pasting shape with decomposition $(A, \{J_i\})$, where the open grid $A$ is witnessed by the map
	$$
	f : \dtr_{\leq d-1}\square[n_1, \dotsc, n_d] \rightarrow A.
	$$
	Then for any box $(x,y) \in I$, we have $(x,y) \in A$ if and only if $x_b = y_b \in \im(f_b)$ for some index $b$. If $(x,y) \notin A$, then there exists an injective map
	$$
	g : \square[m_1, \dotsc, m_d] \rightarrow I,
	$$
	such that $$\im(g_a) = \{t_a \in \im(f_a) \mid x_a \leq t_a \leq y_a\} \cup \{x_a, y_a\}.$$
	In particular, it follows that any injective map 
	$$
	h : \square[p_1, \dotsc, p_d] \rightarrow I
	$$
	with $p_1, \dotsc, p_d \neq 0$, admits an extension 
	$$
	h' : \square[p_1', \dotsc, p_d'] \rightarrow I
	$$
	such that $\im(h'_a) = \{t_a \in \im(f_a) \mid h_a(0) \leq t_a \leq h_a(p_a)\} \cup \im(h_a)$.
\end{lemma}

\begin{proof}
	Similarly as in the proof of \ref{prop.div-pair-boxes}, we inductively define sets $N^k$ of non-degenerate $k$-boxes of $I$. 
	We set $N^0 := B^0(A) \cup \bigcup_i B^0(J_i)$. If $N^{k-1}$ is defined, we define $N^k$ to be the set 
	of non-degenerate $k$-boxes $(x,y)$ of $I$ whose faces are in $N^{k-1}$, and which have one of the following properties: 
	\begin{enumerate}
		\item $(x,y) \in B^k(A)$, 
		\item $(x,y) \notin B^k(A)$, $x_a, y_a \notin \im(f_a)$ if $x_a = y_a$, and there exists an injective map 
	$$
	g : \square[m_1, \dotsc, m_d] \rightarrow I,
	$$
	such that $$\im(g_a) = \{t_a \in \im(f_a) \mid x_a \leq t_a \leq y_a\} \cup \{x_a, y_a\}.$$ 
	\end{enumerate} 
	By construction, the subset of boxes
	$$
	\bigcup_{k=0}^{d} N^k \subset B^d(I)
	$$
 	is closed under faces. We will show that it is also closed under joins and hence defines a subshape $I' \subset I$. 

	Clearly, $N^0$ is closed under joins. Suppose that $N^{k-1}$ is closed under joins. Then we show 
	that the join of adjacent boxes $(x,y)$ and $(x',y')$ in $N^k$ is again in $N^k$. The faces of the join $(x,y')$ must be in $N^{k-1}$, since $N^{k-1}$ is assumed to be closed under joins.
	Thus we must check that the join $(x,y')$ has property (1) or (2).
	If $k = d$, $(x,y)$ and $(x',y')$ must have property (2) since $A$ has no non-degenerate $d$-boxes. If $k < d$, then $(x,y)$ and $(x',y')$ either must have both property (1) or both property (2), 
	since any box $(s,t)$ in $A$ has the property that $s_b = t_b \in \im(f_b)$ for some index $b$. 
	In both cases, it is readily verified that the join of $(x,y)$ and $(x',y')$ again has the same property, hence is in $N^k$.
	 
	We now claim that the resulting subshape $I'$ coincides with $I$. We then obtain a description of the boxes of $I$, from which the first statement of the lemma 
	follows. 
	By assumption, we may write $I$ as the union
	 $$I = A \cup J_1 \cup \dotsb \cup J_m,$$
	Thus it suffices that $A$ and the $J_i$'s are contained in $I'$. The first assertion is clear.
	Suppose that $(x,y)$ is a box in $J_i$, then clearly the desired injective map $g$ exists (with $n_1, \dotsc, n_d \leq 1)$. 
	If $x_b = y_b \in \im(f_b)$ for some index $b$, then we must have that $(x,y)$ is contained in the boundary $\partial J_i \subset A$. Thus, in either case, we have 
	$(x,y) \in I'$. 

	The latter statement of the lemma follows from the preceding assertion and the fact that $A$ has no non-degenerate $d$-boxes.
\end{proof}

\begin{example}
	Consider the pasting shape $I$ and the open subgrid $A \subset I$ of \ref{ex.decomp}. This open subgrid $A$ is part of a decomposition $(A, \{J_i\})$ for $I$.
	The box corresponding to the subshape 
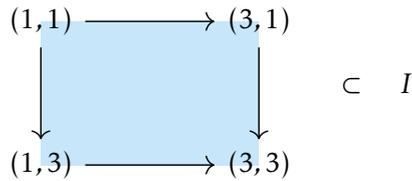
\[
		\begin{tikzcd}[column sep = 8, row sep = 8, ampersand replacement = \&, execute at end picture={
			\scoped[on background layer]
			\fill[pamblue, opacity = 0.5] (a.center) -- (b.center) -- (d.center) -- (c.center);
		}]
			  |[alias=a]|(1,1)\arrow[rr]\arrow[dd] \& \phantom{(2,1)} \& |[alias=b]|(3,1)\arrow[dd]\\
			 \phantom{(1,2)}\& \&  \\
			 |[alias=c]|(1,3)\arrow[rr] \&\& |[alias=d]| (3,3)  
		\end{tikzcd}
		\;\;\;\;\;\subset\;\;\;\;\; I
	\]
	admits a factorization $g : \square[2,2] \rightarrow I$ as in \ref{lemma.decomp-fact-prop}, classified by the subshape 
	
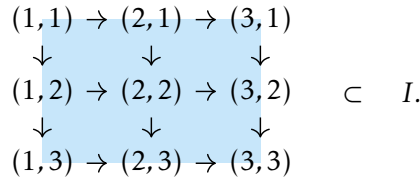
\[
		\begin{tikzcd}[column sep = 8, row sep = 8, execute at end picture={
			\scoped[on background layer]
			\fill[pamblue, opacity = 0.5] (a.center) -- (b.center) -- (d.center) -- (c.center);
		}]
			  |[alias=a]|(1,1)\arrow[r]\arrow[d] & (2,1)\arrow[d]\arrow[r] & |[alias=b]|(3,1)\arrow[d]\\
			 (1,2)\arrow[d]\arrow[r] & (2,2)\arrow[r]\arrow[d] & (3,2) \arrow[d] \\
			 |[alias=c]|(1,3)\arrow[r] & (2,3)\arrow[r] & |[alias=d]| (3,3)   \\
		\end{tikzcd}
		\;\;\;\;\;\subset\;\;\;\;\; I.
	\]
\end{example}

\begin{proposition}\label{prop.closed-entire-admittable}
	Any closed entire subshape $E$ of a $d$-dimensional admittable pasting shape $I$ is again admittable.
\end{proposition}
\begin{proof}
	We proceed by induction on the height $h$ of $I$. If $h = 0$, the statement is clear. Suppose that the statement holds 
	for any admittable pasting shape of height $h$.
	
	Let $I$ be a composable pasting shape of height $h+1$ with decomposition $(A, \{J_i\})$. Write $f : \dtr_{\leq d-1}\square[n_1, \dotsc, n_d] \rightarrow A$ 
	for the map that witnesses $A$ to be an open grid. Denote the bounding box of $E$ by $(x,y)$. Since $E$ is entire, $(x,y)$ must be a non-degenerate $d$-box contained in $I$.
	Consequently,  there exists an extension 
	$$
	g : \square[m_1, \dotsc, m_d] \rightarrow I,
	$$
	with $\im(g_a) = \{t_a \in \im(f_a) \mid x_a \leq t_a \leq y_a\} \cup \{x_a,y_a\}$ on account of \ref{lemma.decomp-fact-prop}.
	Let $B$ be the open subgrid of $I$ determined by $g$. We note that for any inert map 
	$$
	j : \square[1, \dotsc, 1] \rightarrow \square[m_1, \dotsc, m_d],
	$$
	the restriction $gj$ classifies the corners of a closed entire, $d$-dimensional subshape $E_j \subset E$, which has the property that $E \subset J_i$ for some $i$. 
	Thus $E_j$ is admittable by the induction hypothesis. Note that $E_j \cap B = \partial E_j$. Using \ref{lemma.decomp-fact-prop}, one may again show that 
	$$
	E = B \cup \bigcup_{j} E_j.
	$$
	From this it follows that $E$ is admittable with decomposition $(B, \{E_j\})$.
\end{proof}

Using a similar strategy as employed in the proof of \ref{lemma.decomp-fact-prop}, we can  show a stronger factorization property for composable pasting shapes:

\begin{lemma}\label{lemma.decomp-fact-prop-comp}
	Suppose that $I$ is a $d$-dimensional composable pasting shape and let 
	$$
	f : \square[n_1, \dotsc, n_d] \rightarrow I
	$$ be an injective map that carries corners to corners. Then for any box $(x,y) \in I$, there exists an injective map 
	$$
	g : \square[m_1, \dotsc, m_d] \rightarrow I,
	$$
	such that $$\im(g_a) = \{t_a \in \im(f_a) \mid x_a \leq t_a \leq y_a\} \cup \{x_a, y_a\}.$$  
\end{lemma}
\begin{proof}
	We proceed by induction on the dimension $d$. For $d = 0$, the statement is clear. Suppose the statement holds for $(d-1)$-dimensional composable pasting shapes.
	If $I$ has height $0$, then $I$ is a cell, and there is nothing to prove. Suppose that the statement holds for all $d$-dimensional composable pasting shapes 
	that have height $h$.

	Let $I$ be a composable pasting shape of height $h+1$ with decomposition $(A, \{J_i\})$. During this proof, we write $N^k$ for the set of non-degenerate $k$-boxes $(x,y)$ of $I$ with the property that there exists an injective map 
	$$
	g : \square[m_1, \dotsc, m_d] \rightarrow I,
	$$
	such that $$\im(g_a) = \{t_a \in \im(f_a) \mid x_a \leq t_a \leq y_a\} \cup \{x_a, y_a\}.$$ 
	Then the subset of boxes given by
	$$
	\bigcup_{k=0}^{d} N^k \subset B^d(I)
	$$
	is closed under joins and faces, hence constitutes a subshape $I' \subset I$. It now suffices to show that $I'$ contains $J_1, \dotsc, J_n$ and $A$.
	
	Suppose that $(x,y)$ is a box of $A$. Then there exists an index $b$ so that $x_b = y_b \in \im(f_b)$. The map 
	$f$ restricts to a $(d-1)$-truncated map
	$$
	f' : \square[n_1, \dotsc, n_{b-1}, 0, n_{b+1}, \dotsc, n_d] \rightarrow A\cap H.
	$$
	This can be viewed as a map between $(d-1)$-dimensional pasting shapes. As the slice $A\cap H$ must be $(d-1)$-dimensional composable pasting shape, for which the statement already holds, 
	it can now readily be deduced that $(x,y) \in I'$. 

	It remains to show that each box $(x,y)$ of $J_i$ is contained in $I'$. We will write
	$$
	\phi : \dtr_{\leq d-1}\square[p_1, \dotsc, p_d] \rightarrow A,
	$$
	for the injective map that witnesses $A$ to be an open grid. On account of \ref{lemma.decomp-fact-prop},
	there exists an injective map 
	$$
	f' : \square[n_1', \dotsc, n_d'] \rightarrow I
	$$
	such that $\im(f'_a) = \im(\phi_a) \cup \im(f_a)$. Note that the corners of $J_i$ are classified by an inert map 
	$$
	j : \square[1, \dotsc, 1] \rightarrow \square[p_1,\dotsc, p_d],
	$$
	which determines a restriction $r$ of $f'$ that has image given by
	$$\im(r_a) = \{t_a \in \im(f_a) \mid j_a(0) \leq t_a \leq j_a(1)\} \cup \{j_a(0), j_a(1)\} \subset \im(f'_a)$$
	for all $a$. The lemma holds for $J_i$ by assumption as it is composable and has height $h$, hence 
	we may apply the lemma with respect to $(x,y)$ and $r$, yielding the desired factorization $g$ for $(x,y)$.
\end{proof}

\begin{proof}[Proof of \ref{thm.baby-pasting-theorem}]
	We proceed by induction on the height $h$ of $I$. If 
	$I$ has height $0$, there is nothing to show.  Suppose now that the theorem holds for any $d$-dimensional composable pasting shape $I$ of height $h$.
	Let $I$ be a composable pasting shape of height $h+1$ with decomposition $(A, \{J_i\})$. Then there exists an injective map 
	$$
	f : \square[n_1,\dotsc, n_d] \rightarrow I
	$$ 
	such that $\dtr_{\leq d-1}f$ witnesses $A$ to be an open grid. Using \ref{prop.div-pair-boxes} and \ref{lemma.decomp-fact-prop}, 
	one may verify that the vertebrae of $I$ are given by the vertebrae of the $J_i$'s. Thus the 
	induction hypothesis asserts that the vertebrae of 
	$I$ are composable and 
	$$
	 \bigcup_{V \subset I \textit{ vertebra }} V = \bigcup_i J_i.
	$$
	Hence, we must show every box $(x,y)$ of $I$ is contained in the union of the $J_i$'s.
	This readily follows from the factorization property of \ref{lemma.decomp-fact-prop-comp}.
\end{proof}

We may slightly strengthen \ref{thm.baby-pasting-theorem} by introducing so-called \textit{locally composable} pasting shapes. This is 
an auxiliary notion, albeit a pragmatic notion to set up the theory.

\begin{definition}\label{def.locally-composable}
	A $d$-dimensional pasting shape $I$ with finitely many boxes is called \textit{locally composable} if every closed $k$-entire subshape of $I$, with $1 \leq k \leq d$, is a $k$-dimensional admittable pasting shape.
\end{definition}

The following is a direct result of \ref{prop.closed-entire-admittable}:

\begin{proposition}
	A $d$-dimensional pasting shape $I$ is composable if and only if $I$ is locally composable and closed $d$-entire.
\end{proposition}
	
\begin{corollary}\label{cor.baby-pasting-theorem-lc}
	Suppose that $I$ is a $d$-dimensional locally composable pasting shape, then $I$ can be written as the union 
	$$
	I = \bigcup_{V \subset I \text{ closed k-vertebra, } 0 \leq k \leq d} V.
	$$
\end{corollary}
\begin{proof}
	Suppose that $(x,y)$ is a $k$-box of $I$. Then $(x,y)$ determines a $k$-entire subshape $E$ of $I$ which is a $k$-dimensional composable pasting shape by assumption. In view of \ref{thm.baby-pasting-theorem}, 
	$E$ may be written as the union of its (necessarily closed) $k$-vertebrae. Hence $(x,y)$ is contained in the union that is displayed in the statement of the corollary. 
\end{proof}

\begin{example}\label{ex.no-baby-pasting-pinwheel}
	Note that \ref{thm.baby-pasting-theorem} fails to be true in the setting of non-composable pasting shapes. Consider the pinwheel $PW$ that was defined 
	at the start of \ref{ssection.cpshapes}. Then there is a subshape 
	$$PW^\circ \subset PW$$ 
	whose boxes are given by $B^2(PW^\circ) = B^2(PW) \setminus \{((0,0), (3,3))\}$. It is readily verified 
	that $PW^\circ$ is locally composable. Since every $1$-vertebra of $PW^\circ$ is contained in a $2$-vertebra of $PW^\circ$, 
	\ref{cor.baby-pasting-theorem-lc} asserts that 
	$$
	PW^\circ = \bigcup_{V \subset PW^\circ \text{ $k$-vertebra, } 0 \leq k \leq 2} V = \bigcup_{V \subset PW^\circ \text{ vertebra}} V.
	$$
	But the vertebrae of $PW^\circ$ are precisely those of $PW$. Since $PW \neq PW^\circ$, this shows that \ref{thm.baby-pasting-theorem} fails for the pinwheel.
\end{example}

\subsection{The pasting theorem}\label{ssection.ptheorem}
Having developed the basic theory of pasting shapes, we are now able to formulate the pasting theorem which will be proven in the next section.

\begin{definition}
Let $I$ be a $d$-dimensional pasting shape. Then a \textit{covering of $I$} is a collection $I_1, \dotsc, I_n$ of subshapes of $I$, such that 
\begin{enumerate}
	\item every closed $k$-entire subshape of $I_i$, with $0 \leq k \leq d$, is entire in $I$,
	\item every closed $k$-vertebra of $I$, with $0 \leq k \leq d$, is contained in some $I_i$.
\end{enumerate}
\end{definition}

\begin{example}\label{ex.coverings}
	We have the following canonical examples of coverings:
	\begin{enumerate}
	\item Every composable pasting shape is covered by its vertebrae.
	\item In general, if $I$ is a locally composable pasting shape, then the collection of closed $k$-vertebrae of $I$, where $k$ ranges over all dimensions $0, \dotsc, d$, 
	is a covering of $I$.
	\item If $I$ is an admittable pasting shape with decomposition $(A, \{J_i\})$ then $A$ together with the $J_i$'s form a covering for $I$.
	\end{enumerate}
\end{example}

We would like to phrase our pasting theorem for $d$-uple Segal spaces 
as a statement about the nerve functor preserving certain colimits. 
To this end, we will need to understand how unions of pasting shapes may be described as colimits.
This can be done as follows. Suppose that $S$ is a set of $d$-dimensional pasting shapes. Then we may consider the set of pasting shapes given by the non-empty intersections 
$$
\textstyle \mathcal{P}(S) := \{\bigcap_{I \in T}I \mid T \subset S\} \setminus \{\emptyset\}
$$
which becomes a (finite) poset under inclusion. Now, one may readily verify that the colimit 
of the tautological functor $$Q : \mathcal{P}(S) \rightarrow \Shape^d : J \mapsto J$$ is given by 
the union $\bigcup_{I \in S} I$.  

\begin{theorem}[The pasting theorem]\label{thm.pasting-theorem}
Suppose that $I_1, \dotsc, I_n$ is a covering of a $d$-dimensional locally composable pasting shape $I$. Then $I$ can be written as
$$
I = \bigcup_{i=1}^n I_i,
$$
and this union is preserved by the nerve functor $[-] : \Shape^d \rightarrow \Cat^d(\S)$, so that the comparison map
$$
\textstyle \colim_{J \in \mathcal{P}(\{I_1, \dotsc, I_n\})} [J] \rightarrow [I]
$$
is an equivalence of $d$-uple Segal spaces.
\end{theorem}

\begin{remark}\label{rem.pasting-theorem-union}
The pasting theorem can also be rephrased as follows. Suppose that $I_1, \dotsc, I_n$ is a covering of a locally composable pasting shape $I$.
Since $\Cat^d(\S)$ is a reflective subcategory of $\fun(\Delta^{\op, \times d}, \S)$, we may compute and reflect the colimit appearing in the statement in the latter category 
so that the statement of the theorem is equivalent to saying that the map
$$
\textstyle \colim_{J \in \mathcal{P}(\{I_1, \dotsc, I_n\})} [J] \rightarrow [I]
$$
of $d$-uple simplicial spaces is a Segal equivalence.

More is true: the colimit on the left may be computed in $\fun(\Delta^{\op, \times d}, \cat{Set})$. The colimit is then given by the ordinary union so that the theorem equivalently asserts that the inclusion
$$
\textstyle \bigcup_{i=1}^n [I_i] \rightarrow [I]
$$
of $d$-uple simplicial sets is a Segal equivalence. To deduce that we can compute the colimit in $\fun(\Delta^{\op, \times d}, \cat{Set})$, we can proceed level-wise and apply similar 
reasoning as in the proof of \ref{prop.pushout-seq}. We have to show that the homotopy colimit (w.r.t.\ the Kan-Quillen model structure on $\mathrm{sSet}$) of the diagram
$$
	\textstyle [Q]_{k_1,\dotsc, k_n} : \mathcal{P}(\{I_1, \dotsc, I_n\}) \rightarrow \mathrm{sSet} : J \mapsto 
		[J]_{k_1,\dotsc, k_n}.
$$
is given by the ordinary colimit for all $k_1,\dotsc, k_n \geq 0$. 

We will do so by making use of the theory of \textit{Reedy model structures} (see \cite[Section 15.1]{Hirsch} or \cite{RiehlVerity}).
Recall that $\mathcal{P}(\{I_1, \dotsc, I_n\})$ is a finite poset, so it comes with a canonical Reedy structure (see \cite[Example 2.3]{RiehlVerity}). So the homotopy colimit of a general 
diagram $F : \mathcal{P}(\{I_1, \dotsc, I_n\}) \rightarrow \mathrm{sSet}$ may be computed as the colimit of $F$ 
if $F$ is Reedy cofibrant.
By definition, the diagram $F$ is Reedy cofibrant if and only if for any  the {latching map} for 
$F$ at each $J \in \mathcal{P}(\{I_1, \dotsc, I_n\})$
$$
	\textstyle \colim_{J' \subsetneq J \in \mathcal{P}(\{I_1, \dotsc, I_n\})} F(J') \rightarrow F(J).
$$
is a monomorphism of simplicial sets. In the case that $F = [Q]_{k_1,\dotsc, k_n}$, the latching map at such $J$ is given by the canonical map 
$$
	\textstyle \bigcup_{J' \subsetneq J \in \mathcal{P}(\{I_1, \dotsc, I_n\})}  [J']_{k_1,\dotsc, k_n} \rightarrow [J]_{k_1,\dotsc, k_n},
$$
which is a monomorphism.
\end{remark}

If $I$ is a composable pasting shape, then the vertebrae of $I$ give a canonical covering of $I$ to which the above theorem applies. 

\begin{definition}
	Let $I$ be a $d$-dimensional composable pasting shape. 
	The \textit{spine} of $[I]$ is the $d$-uple simplicial subset of $[I]$ given by the union 
	$$
	\Sp[I] := \bigcup_{V \subset I \textit{ vertebra}} [V] \subset [I].
	$$
	We will call the resulting inclusion $\Sp[I] \rightarrow [I]$ the \textit{spine inclusion for $I$}.
\end{definition}

\begin{corollary}\label{cor.spine-inclusion}
	The spine inclusion for a composable pasting shape $I$
	is a Segal equivalence.
\end{corollary}

\begin{remark}\label{rem.spine-ordinary}
	The spine inclusions defined above generalize the spine inclusions that were introduced in \ref{ssection.cat-objs}. 
	Consider the grid $ \square[n_1, \dotsc, n_d]$ so that $n_1, \dotsc, n_d \neq 0$.  We have 
	exhibited the collection $\mathcal{U}$ of vertebrae of such a standard grid in \ref{ex.trivial-grid}.
	There is a canonical isomorphism of categories
	$$
	\mathbb{G}^{\times d}/([n_1], \dotsc, [n_d]) \rightarrow \mathcal{P}(\mathcal{U})
	$$
	that carries a tuple $(i_a : [k_a] \rightarrow [n_a])$ to the image of 
	the corresponding map $i : \square[k_1, \dotsc, k_d] \rightarrow \square[n_1, \dotsc, n_d]$ of pasting shapes.
	Using the identifications of \ref{cor.cosimp-objs}, we obtain a commutative square
	\[
		\begin{tikzcd}
			\colim_{([k_1],\dotsc, [k_d]) \in \mathbb{G}^{\times d}/([n_1], \dotsc, [n_d])} \Delta[k_1, \dotsc, k_d] \arrow[r, "\simeq"] \arrow[d]& \colim_{J \in \mathcal{P}(\mathcal{U})} [J] \arrow[d] \\
		\Delta[n_1, \dotsc, n_d] \arrow[r, "\simeq"] & {[\square[n_1, \dotsc, n_d]]}
		\end{tikzcd}
	\]
	in $\fun(\Delta^{\op, \times d}, \S)$. The right arrow is precisely the spine inclusion for the grid $\square[n_1,\dotsc, n_d]$ on account of \ref{rem.pasting-theorem-union}. Thus the spine inclusion for the standard grid is a Segal equivalence by definition.
\end{remark}

\begin{example}\label{ex.no-pasting-pinwheel}
	Using a similar argument as in \ref{ex.no-baby-pasting-pinwheel}, one sees that the spine inclusion for $PW$ is not a Segal equivalence. 
	Namely, it factors as 
	$$
	\Sp[PW] = \Sp[PW^\circ] \rightarrow [PW^\circ] \rightarrow [PW].
	$$ 
	The pasting theorem asserts that the first inclusion is a Segal equivalence. By the 2-out-of-3 principle, the spine inclusion 
	for $PW$ is a Segal equivalence precisely when the inclusion $[PW^\circ] \rightarrow [PW]$ is a Segal equivalence. But this 
	cannot be the case since the nerve functor is fully faithful and $PW \neq PW^\circ$.
\end{example}

\begin{corollary}\label{cor.composite-contractible-choice}
	Consider a $d$-dimensional pasting shape $I$ and a $d$-uple Segal space $X$. 
	Suppose that $I$ is a composable shape and that we have a family of maps $$f_V : [V] \rightarrow X, \quad V \subset I \text{ vertebra},$$ which 
	are compatible in the sense that 
	it defines a map $ \Sp[I] \rightarrow X$ 
	of $d$-uple simplicial spaces. Then the space $C$ of composites defined by the pullback square
	\[
		\begin{tikzcd}
			C \arrow[r]\arrow[d] & \map_{\Cat^d(\S)}([I], X) \arrow[d] \\
			\{(f_V)\} \arrow[r] & \map_{\fun(\Delta^{\op, \times d}, \S)}(\Sp[I], X),
		\end{tikzcd}
	\]
	is contractible.
\end{corollary}

\section{Proof of the pasting theorem}\label{section.proof}

Our strategy for proving the pasting theorem is by proceeding inductively on the dimension of the pasting shapes. We introduce the following auxiliary notion:

\begin{definition}
	We call a $d$-dimensional pasting shape $I$ \textit{good} if the map 
	$$
	\Sp[I] \cup \dtr_{\leq d-1}[I] \rightarrow [I]
	$$
	is a Segal equivalence.
\end{definition}

\begin{remark}\label{rem.generalized-spine}
	We may describe the $d$-uple simplicial subset of $[I]$ appearing in the domain of the map above, alternatively by
	$$
	\Sp[I] \cup \dtr_{\leq d-1}[I] = \bigcup_{\substack{i : \square[1, \dotsc, 1] \rightarrow I \text{ injective},  \\ 
	\im(i) \text{ is contained in a (closed) vertebra of $I$}}} i[\square[1, \dotsc, 1]] \cup \dtr_{\leq d-1}[I].
	$$
\end{remark}

The following theorem is the crucial ingredient that makes the induction work:

\begin{theorem}\label{thm.spine-admittable}
Every $d$-dimensional admittable pasting shape is good.
\end{theorem}

The proof of this theorem is deferred to the end of this section. Instead, we will now demonstrate how the pasting theorem follows from this fact. We distill two 
technical steps in this demonstration as lemmas. They will also be of use later.

\begin{lemma}\label{lemma.locally-good}
	Suppose that $I$ is a $d$-dimensional pasting shape so that any closed entire subshape of $I$ 
	is good.
	If $I_1, \dotsc, I_n$ is a collection 
	of subshapes of $I$ 
	with the property that any closed entire subshape of $I_i$ is an entire subshape of $I$, then the inclusion
	$$
	\textstyle \bigcup_i\Sp[I_i] \cup \dtr_{\leq d-1}[I_i] \rightarrow \bigcup_i [I_i]
	$$
	of $d$-uple Segal sets is a Segal equivalence.
\end{lemma}
\begin{proof}
	During this proof, we will call a subshape $J$ of $I$ \textit{total} if it has the property that any closed entire subshape of $J$ is again an entire subshape of $I$. 
	We consider the set $S$ of pairs of $d$-uple simplicial subsets of $[I]$ 
	$$
	(\Sp[J] \cup \dtr_{\leq d-1}[J], [J])
	$$
	where $J$ is $(d-1)$-truncated or a total subshape of $I$. One readily verifies that $S$ meets condition (1) of \ref{lemma.unions-seq}, 
	so that we can reduce to showing that the inclusion $$
	i : \Sp[J] \cup \dtr_{\leq d-1}[J] \rightarrow [J],
	$$
	is a Segal equivalence for $J$ $(d-1)$-truncated or total. If $J$ is $(d-1)$-truncated, then this inclusion is the identity. 
	Hence, we may assume that $J$ is total.
	Let $E_1, \dotsc, E_m$ be all closed entire subshapes of $J$. Then the inclusion $i$ is precisely the inclusion
	\[
		\textstyle	i : \bigcup_{i=1}^m \Sp[E_i] \cup \dtr_{\leq d-1}[I] \rightarrow \bigcup_{i=1}^m [E_i] \cup \dtr_{\leq d-1}[I].
	\]
	This inclusion is induced by the map
	\[
		\textstyle	j : \bigcup_{i=1}^m \Sp[E_i] \cup \dtr_{\leq d-1}[E_i] \rightarrow \bigcup_{i=1}^m [E_i] \cup \dtr_{\leq d-1}[E_i]
	\]
	under taking pushouts along the inclusion $\bigcup_{i=1}^m \dtr_{\leq d-1}[E_i] \rightarrow \dtr_{\leq d-1}[I]$, hence it suffices to show that 
	$j$ is a Segal equivalence. 
	We can make use of \ref{lemma.unions-seq} once again to verify this. To this end, consider the subset $S' \subset S$ which 
	now contains the pairs associated to $(d-1)$-truncated or closed entire subshapes of $I$. Then $S'$ again meets 
	condition (1) of \ref{lemma.unions-seq}. By assumption, $S'$ also meets (2), so that it follows that $j$ is a Segal equivalence.
\end{proof}

\begin{lemma}\label{lemma.pasting-theorem}
	Fix an integer $0 \leq d' \leq d$.
	Suppose that $I$ is a $d$-dimensional pasting shape that has the property that any closed $k$-entire subshape of $I$ with $k \leq d'$, viewed as a $k$-dimensional pasting shape,
	is good.
	If $I_1, \dotsc, I_n$ is a cover of $I$, then the inclusion
	$$
	\textstyle \bigcup_i \dtr_{\leq k}[I_i] \rightarrow \dtr_{\leq k}[I]
	$$
	of $d$-uple Segal sets is a Segal equivalence for every $0 \leq k \leq d'$.
\end{lemma}
\begin{proof}
	Throughout, let us write 
	$$
	\textstyle j : \bigcup_i [I_i] \rightarrow [I]
	$$ 
	for the inclusion.
	We will show that each truncated inclusion $\dtr_{\leq k}j$ is a Segal equivalence by induction on $0 \leq k \leq d'$. For $k = 0$, the statement is clear. Suppose that the statement holds for $k-1 < d'$. 
	Then we may consider the factorization 
	$$
	\textstyle \dtr_{\leq k}\bigcup_i[I_i] \rightarrow \dtr_{\leq k}\bigcup_i[I_i] \cup \dtr_{\leq k-1}[I] \xrightarrow{j_k} \dtr_{\leq k}[I]
	$$ 
	of the truncated inclusion $\dtr_{\leq k}j$. Note that the map 
	on the left is a pushout of the inclusion map $\dtr_{\leq k -1}j$, which is a Segal equivalence by assumption.
	Hence, this map is again a Segal equivalence.
	In light of the 2-out-of-3 principle, the truncation $\dtr_{\leq k} j$ is a Segal equivalence precisely if the inclusion $j_k$ is a Segal equivalence.

	In order to prove that $j_k$ is a Segal equivalence, we will exploit \ref{prop.desc-trunc} which asserts that $j_k$ can be written as a union of inclusions
	$$
	\textstyle \bigcup_{H \in \mathscr{H}_k} \bigcup_i {[I_i \cap H]} \cup \dtr_{\leq k-1}[I]\rightarrow \bigcup_{H \in \mathscr{H}_k}[I \cap H].
	$$
	Hence, \ref{lemma.unions-seq} implies that it suffices to verify that for any $H\in \mathscr{H}_k$, the inclusion 
	$$
	\textstyle \bigcup_i {[I_i \cap H]} \cup \dtr_{\leq k-1}[I] \rightarrow [I\cap H] \cup \dtr_{\leq k-1}[I]
	$$
	is an equivalence.
	In turn, this map is induced by the map 
	$$
	\textstyle j_k^H : \bigcup_i {[I_i \cap H]} \cup \dtr_{\leq k-1}[I \cap H] \rightarrow [I\cap H],
	$$
	under taking pushouts along $\dtr_{\leq k-1}[I\cap H] \rightarrow \dtr_{\leq k-1}[I]$, so that 
	it is enough to demonstrate that $j_k^H$ is an equivalence.

	To this end, we note that the slices $I_0 \cap H, \dotsc, I_n \cap H \subset I\cap H$ correspond to $k$-dimensional pasting shapes $J_0, \dotsc, J_n \subset J$ 
	under the embedding $i_h$ of \ref{rem.hyperplane-embedding} for a suitable choice of $k$-tuple $h$. It readily follows from the definitions that $J_0, \dotsc, J_n$ again form a covering for 
	$J$. The compatibility of $i_h$ with the nerve functor (see \ref{rem.hyperplane-embedding}), 
	allows us to identify $j_k^H$ with the image of the inclusion
	$$
	\textstyle \bigcup_i {[J_i]} \cup \dtr_{\leq k-1}[J] \rightarrow [J]
	$$
	of $k$-uple simplicial sets under the functor $p_h^*$ (defined in \ref{ssection.cat-objs}). 
	The right adjoint of $p_h^*$ preserves iterated Segal spaces on account of \ref{prop.cat-k-d-incl}.
	Hence, the functor $p_h^*$ will preserve Segal equivalences. Thus it suffices to show that 
	the above inclusion of $k$-uple simplicial sets is a Segal equivalence. 
	This fits in a commutative square of inclusions
	\[
		\begin{tikzcd}
			\bigcup_i {\Sp[J_i]} \cup \dtr_{\leq k-1}[J] \arrow[r]\arrow[d] & \bigcup_i {[J_i]} \cup \dtr_{\leq k-1}[J] \arrow[d] \\
			\Sp[J] \cup \dtr_{\leq k-1}[J]  \arrow[r] & {[J]}.
		\end{tikzcd}
	\]
	The left vertical map must be the identity because $J_0, \dotsc, J_n$ cover $J$. Moreover, \ref{thm.spine-admittable} implies that the bottom map is a Segal equivalence. 
	The top map is induced by the inclusion 
	$$
	\textstyle \bigcup_i {\Sp[J_i]} \cup \dtr_{\leq k-1}[J_i] \rightarrow \bigcup_i {[J_i]}
	$$
	under taking pushouts along the map $\bigcup_i \dtr_{\leq k-1}[J_i] \rightarrow [J]$. In light of our assumptions, 
	this inclusion will be an equivalence by \ref{lemma.locally-good},
	so the desired result now follows from the 2-out-of-3 principle. 
\end{proof}

We are now ready to give a proof of \ref{thm.pasting-theorem} assuming that \ref{thm.spine-admittable} holds:

\begin{proof}[Proof of \ref{thm.pasting-theorem}]
	The fact that $I$ decomposes as $I = \bigcup_i I_i$ follows directly from \ref{cor.baby-pasting-theorem-lc} and the definition of a covering. One can retrieve 
	the main result 
	from the observations in \ref{rem.pasting-theorem-union} and \ref{lemma.pasting-theorem}.
\end{proof}

\subsection{Prototypical good pasting shapes}

It remains to prove \ref{thm.spine-admittable}. We will exhibit a few prototypical examples of good pasting shapes.

\begin{proposition}\label{prop.stgrid-good}
	The standard grids $\square[n_1, \dotsc, n_d]$ are good.
\end{proposition}
\begin{proof}
	We proceed by induction on the dimension $d$. If $d = 0$, then there is nothing to show. Suppose that the proposition 
	holds for the $k$-dimensional standard grids with $k \leq d-1$. Consider a $d$-dimensional standard grid $\square[n_1, \dotsc, n_d]$. If $n_a = 0$ for some index $a$, there is nothing to show. 
	Hence, we may assume that $n_1, \dotsc, n_d \neq 0$. In what follows, we will suppress the notation of this tuple for brevity. The spine inclusion of $\square$ factors as
	$$
	\Sp[\square] \xrightarrow{j} \Sp[\square] \cup \dtr_{\leq d-1}[\square] \rightarrow [\square].
	$$
	We have already demonstrated in \ref{rem.spine-ordinary} that the composite is a Segal equivalence. Each closed $k$-entire subshape of $\square$ 
	is isomorphic to a $k$-dimensional standard grid, hence it is good for $k \leq d-1$ on account of the induction hypothesis.
	An application of \ref{lemma.pasting-theorem} with respect to the covering given by the vertebrae of $\square$ 
	now gives that the truncated inclusion
	$$\dtr_{\leq d-1}\Sp[\square] \rightarrow \dtr_{\leq d-1}[\square]$$
	is a Segal equivalence. As $j$ is a pushout of this map, this proves the statement in light of the 2-out-of-3 principle.
\end{proof}

\begin{corollary}
	Any grid is good.
\end{corollary}
\begin{proof}
	Let $A$ be a $d$-dimensional grid. If $A$ is open, then \ref{rem.generalized-spine} implies that there is nothing to show. Hence, we may assume that 
	there exists a map 
	$$
	f : \square[n_1, \dotsc, n_d] \rightarrow A,
	$$
	with $n_1, \dotsc, n_d \neq 0$, that witnesses $A$ to be a grid.  
	When we combine the description given in \ref{rem.generalized-spine}, the description of the vertebrae of $A$ 
	in \ref{ex.trivial-grid}, and the fact that $f$ is $d$-shaping, we deduce that we have a pushout square 
	\[
		\begin{tikzcd}	
			\Sp[\square] \cup \dtr_{\leq d-1}[\square] \arrow[r]\arrow[d] & \Sp[A] \cup \dtr_{\leq d-1}[A] \arrow[d] \\
			{[\square]} \arrow[r] & {[A]}.
		\end{tikzcd}
	\]
	The left map is a Segal equivalence on account of \ref{prop.stgrid-good}, hence the result follows.
\end{proof} 

The following is also an important example of a good pasting shape:

\begin{definition}
	Suppose that 
	$$
	i : \square[1,\dotsc, 1]\rightarrow \square[n_1, \dotsc, n_d]
	$$
	is an injective map of $d$-dimensional pasting shapes.
	 Then we define the (admittable) pasting shape 
	$$\boxdot[i] \subset \square[n_1, \dotsc,n_d]$$
	to be the subshape of $\square[n_1, \dotsc, n_d]$ that consists of all those boxes $(x,y)$ such that 
	there exists an index $a$ with $x_a \leq y_a \leq i_a(0)$ or $i_a(1) \leq x_a \leq y_a$ (one readily checks that these boxes are closed 
	under faces and joins).
\end{definition}

\begin{example}\label{ex.boxdot}
	The following is a picture for $\boxdot[i]$, with maps $i_1 : [1] \rightarrow [6], i_2 : [1] \rightarrow [4]$ starting and ending at $2, 5$ and $1,3$ respectively:
	\[
	\begin{tikzcd}[execute at end picture={
		\scoped[on background layer]
		\fill[pamblue, opacity = 0.5] (a1.center) -- (a2.center) -- (b4.center) -- (b1.center);
		\scoped[on background layer]
		\fill[pamblue, opacity = 0.5] (b1.center) -- (b2.center) -- (c2.center) -- (c1.center);
		\scoped[on background layer]
		\fill[pamblue, opacity = 0.5] (b3.center) -- (b4.center) -- (c4.center) -- (c3.center);
		\scoped[on background layer]
		\fill[pamblue, opacity = 0.5] (c1.center) -- (c4.center) -- (d2.center) -- (d1.center);
	}]
		|[alias=a1]|\cdot \arrow[r]\arrow[d] & \cdot \arrow[r]\arrow[d] & \cdot \arrow[r] \arrow[d] & \cdot \arrow[r]\arrow[d] & \cdot \arrow[r]\arrow[d]  & \cdot \arrow[r]\arrow[d] &|[alias=a2]|\cdot \arrow[d]  \\
		|[alias=b1]|\cdot \arrow[r]\arrow[d] & \cdot \arrow[r]\arrow[d] & |[alias=b2]| \cdot \arrow[r] \arrow[d] & \cdot \arrow[r] & \cdot \arrow[r]  & |[alias=b3]| \cdot \arrow[r]\arrow[d] &|[alias=b4]|\cdot \arrow[d]  \\
		\cdot \arrow[r]\arrow[d] & \cdot \arrow[r]\arrow[d] & \cdot \arrow[d] &  & {}   & \cdot \arrow[r]\arrow[d] &\cdot \arrow[d]  \\
		|[alias=c1]| \cdot \arrow[r]\arrow[d] & \cdot \arrow[r]\arrow[d] &|[alias=c2]|  \cdot \arrow[r] \arrow[d] & \cdot \arrow[r]\arrow[d] & \cdot \arrow[r]\arrow[d]  &|[alias=c3]|\cdot \arrow[r]\arrow[d] & |[alias=c4]|\cdot \arrow[d]  \\
		|[alias=d1]|\cdot \arrow[r] & \cdot \arrow[r] & \cdot \arrow[r] & \cdot \arrow[r] & \cdot \arrow[r]  & \cdot \arrow[r]&|[alias=d2]|\cdot
	\end{tikzcd}
	\]
\end{example}

\begin{proposition}\label{prop.standard-fillable-good}
	Suppose that $i : \square[1, \dotsc, 1] \rightarrow \square[n_1,\dotsc,n_d]$ is an injective map, then the
	pasting shape $\boxdot[i]$ is good.
\end{proposition}

Let us first make the necessary preparations for its proof.

\begin{definition}
	Let $i : \square[1, \dotsc, 1]\rightarrow \square[n_1, \dotsc, n_d]$ be an injective map.
	For a sign $\sigma \in \{-, +\}$  and index $1 \leq a \leq d$, we define 
	$$M^\sigma_a[i] \subset \square[n_1,\dotsc, n_d]$$ to be the subshape that 
	contains all boxes whose vertices $x$ satisfy
	$$
	\begin{cases}
		x_a \leq i_a(0) & \text{if $\sigma = -$}, \\
		x_a \geq i_a(1) & \text{if $\sigma = +$}.
	\end{cases}
	$$ 
	Note that $M^\sigma_a[i]$ is contained in $\boxdot[i]$.
	Moreover, we define 
	$$
	A[i] \subset \square[n_1,\dotsc, n_d]
	$$
	to be to be the largest subshape of $\square[n_1,\dotsc, n_d]$ whose vertices $x$ satisfy $i_a(0) \leq x_a \leq i_a(1)$ for all indices $a$.
\end{definition}

\begin{lemma}\label{lemma.max-grids-boxdot}
	Let $i : \square[1,\dotsc,1]\rightarrow \square[n_1, \dotsc, n_d]$ be an injective map.
	Then any map of pasting shapes $\square[k_1,\dotsc, k_d]\rightarrow \boxdot[i]$ factors through one of 
	the subshapes  $M^{\sigma}_a[i]$.
\end{lemma}
\begin{proof}
	Consider a map of pasting shapes $f : \square[k_1, \dotsc, k_d] \rightarrow \boxdot[i]$ and suppose to the contrary that $f$ does not factor through a single 
	$M^\sigma_a[i]$. Then for any $a$, there exist a $t_a$ and $t'_a$ in the image of $f_a$ such that $t_a > i_a(0)$ and $t'_a < i_a(1)$. Define $x_a := \min(t_a, t'_a)$ and $y_a := \max(t_a, t'_a)$. 
	Then $(x,y)$ is in the image of $f$, and hence must be a box in $\boxdot[i]$. Thus there is an index $a$ such that $x_a \leq y_a \leq i_a(0)$ or $i_a(1) \leq x_a \leq y_a$. 
	This is a contradiction as $i_a(1) > t'_a \geq x_a$ and $y_a \geq t_a > i_a(0)$. 
\end{proof}

\begin{proof}[Proof of \ref{prop.standard-fillable-good}]
	For brevity, we drop the notation of the map $i$ in what follows.
	Note that \ref{lemma.max-grids-boxdot} precisely asserts that the inclusion $$\Sp[\boxdot] \cup \dtr_{\leq d-1}[\boxdot] \rightarrow [\boxdot]$$ is given by 
	$$
	\textstyle \bigcup_{\sigma \in \{+,-\}, 1\leq a \leq d} \Sp[M^\sigma_a] \cup \dtr_{\leq d-1}[M^\sigma_a] \rightarrow \bigcup_{\sigma \in \{+,-\}, 1\leq a \leq d} [M^\sigma_a].
	$$
	We may apply \ref{lemma.locally-good} to show that this is a Segal equivalence. To this end, we must show that 
	any closed entire subshape $E$ of $\boxdot$ is good. 
	Indeed, \ref{lemma.max-grids-boxdot} implies that $E$ must be contained in some $M^\sigma_a$. In this case, $M^\sigma_a$ is (isomorphic to) a standard grid, 
	and hence $E$ is a standard grid as well. Thus $E$ is good on account of \ref{prop.stgrid-good}.
\end{proof}

\subsection{Fillable shapes}

Until now we have dealt with concrete admittable pasting shapes and shown that they are good. In order to deal with arbitrary 
admittable pasting shapes and show that they are good as well, we need some machinery to bootstrap the results we have obtained so far. We will 
develop this in the following subsection. 

\begin{definition}
	A division pair $(K,J)$ for an  admittable pasting shape $I$ is called \textit{good} if $K$ and $J$ are both good. 
\end{definition}

The following is a result of the main theorem, \ref{thm.division-pair-fillables}, of this final subsection:

\begin{corollary}\label{cor.division-pair-good}
	Suppose that $(K,J)$ is a good division pair for $I$ so that the corners of $J$ are given by $(\alpha_1, \omega_1), \dotsc, (\alpha_d, \omega_d)$. Then $I$ is good if any injective map 
	$$
	f : \square[m_1,\dotsc, m_d] \rightarrow I,
	$$
	with $m_1, \dotsc, m_d \neq 0$,
	admits an extension 
	$$
	g : \square[n_1, \dotsc, n_d] \rightarrow I
	$$
	such that for each index $a$, $\im(f_a) \subset \im(g_a)$, and if 
	$g_a(0) \leq \alpha_a \leq g_a(n_a)$ or $g_a(0) \leq \omega_a \leq g_a(n_a)$, then, respectively, $\alpha_a \in \im(g_a)$ or $\omega_a \in \im(g_a)$.
\end{corollary}

Let us first show how we can finish the proof of the pasting theorem using this result.

\begin{proof}[Proof of \ref{thm.spine-admittable}]
Let $I$ be a $d$-dimensional admittable pasting shape. If $I$ is height $0$, then $I$ is a cell and there is nothing to prove. Suppose that the statement holds 
for any admittable pasting shape $I$ of height $h$.

Let $I$ now be an admittable pasting shape of height $h+1$ with decomposition $(A, \{J_i\}_{i=1}^m)$, where $A$ is determined by a map 
$f : \dtr_{\leq d-1}\square[n_1,\dotsc,n_d] \rightarrow A$. For $0 \leq k \leq m$, consider the subshape 
$$
I_{k} := A \cup \bigcup_{i=1}^k J_i \subset I.
$$
Then it suffices to show that $I_k$ is good for all $I$. We proceed by induction on $k$. Trivially, the pasting shape $I_0 = A$ is good since it is $(d-1)$-truncated. Suppose that $I_{k-1}$ is good.  
The subshape  $J_k$ has height $h$, hence is good by assumption. It follows that the pair 
$(I_{k-1}, J_k)$ is a good division pair for $I_k$. Note that $I_k$ has a decomposition given by $(A, \{J_0, \dotsc, J_k, \partial J_{k+1}, \dotsc, \partial J_m\})$. Thus, in light of the factorization property of \ref{lemma.decomp-fact-prop}, 
any map $h : \square[m_1,\dotsc, m_d] \rightarrow I_k$, with $m_1, \dotsc, m_d\neq 0$, admits an extension 
$ g : \square[p_1,\dotsc, p_d] \rightarrow I_k$ such that 
$$
\im(g_a) = \{t_a \in \im(f_a) \mid h_a(0) \leq t_a \leq h_a(m_a)\} \cup \im(h_a).
$$
This is a desired extension $g$ of $h$ as in \ref{cor.division-pair-good}, hence it follows from this result that $I_k$ is good.
\end{proof}

\begin{proposition}\label{prop.division-pair-spine}
Suppose that $(K,J)$ is a division pair for $I$. Then the spine inclusion for $I$ factors 
as the composite of inclusions
$$
\Sp[I] \cup \dtr_{\leq d-1}[I]  \rightarrow [K] \cup [J] \cup \dtr_{\leq d-1}[I] \rightarrow [I],
$$
and the left map is a Segal equivalence if $(K,J)$ is good.
\end{proposition}
\begin{proof}
Suppose that all the distinct vertebrae of $K$ are given by $V_1, \dotsc, V_k, \partial J$, and the vertebrae of $J$ are given by $W_1, \dotsc, W_l$. 
Then the vertebrae of $I$ are given by $V_1, \dotsc, V_k, W_1, \dotsc, W_l$. Hence, the spine inclusion factors through the inclusion 
$[K] \cup [J] \cup \dtr_{\leq d-1}[I] \subset [I]$. Note that we have pushout diagrams 
\[
	\begin{tikzcd}
		\Sp[K] \cup \dtr_{\leq d-1}[K] \arrow[r] \arrow[d] & \Sp[I] \cup \dtr_{\leq d-1}[I] \arrow[d] \\
		{[K]} \arrow[r] & \Sp[I] \cup {[K]} \cup \dtr_{\leq d-1}[I],
	\end{tikzcd}
\]
and
\[
	\begin{tikzcd}
		\Sp[J] \cup \dtr_{\leq d-1}[J] \arrow[r] \arrow[d] & \Sp[I] \cup [K] \cup \dtr_{\leq d-1}[I] \arrow[d] \\
		{[J]} \arrow[r] & {[K]} \cup {[J]} \cup \dtr_{\leq d-1}[I].
	\end{tikzcd}
\]
If $(K,J)$ is good, the left vertical maps in these diagrams are Segal equivalences, hence, the composite of the right vertical maps must then be a Segal equivalence as well.
\end{proof}

\begin{example}\label{ex.division-pair-boxdot}
	The division pair $(\boxdot[i], A[i])$ for $\square[n_1,\dotsc,n_d]$  associated to an injective map $$i:\square[1, \dotsc, 1] \rightarrow \square[n_1,\dotsc, n_d],$$ is good. Since 
	$\square[n_1,\dotsc, n_d]$ is good, \ref{prop.division-pair-spine} 
	implies that the inclusion 
	$$
	[\boxdot[i]] \cup [A[i]] \cup \dtr_{\leq d-1}[\square[n_1,\dotsc, n_d]] \rightarrow [\square[n_1,\dotsc, n_d]]
	$$
	is a Segal equivalence.
\end{example}

\begin{definition}
	Suppose that $(K,J)$ and $(K', J')$ are division pairs for, respectively, $d$-admittable pasting shapes $I$ and $I'$.
	Then a \textit{map of division pairs} $$f : (K,J) \rightarrow (K', J')$$
	is a map of pasting shapes $f : I \rightarrow I'$ such that $$f^{-1}(J') = J, \quad \text{ and, } \quad f^{-1}(K') \cup \dtr_{\leq d-1}I = K \cup \dtr_{\leq d-1}I.$$
\end{definition}

In view of \ref{prop.division-pair-spine}, we have to show that the map appearing on the right is a Segal equivalence. This amounts to filling simplices (or rather: nerves of subshapes) that 
do not appear in the domain of this map, using pushouts of Segal equivalences. We single out a class of pasting shapes that we will be able to fill: 

\begin{definition}\label{def.fillables}
	Suppose that $(K,J)$ is a division pair for $I$. We call a subshape $F \subset I$ 
	\textit{$(K,J)$-fillable} if $F$ is $(d-1)$-truncated or there exists a $d$-shaping map $$f : \square[n_1,\dotsc, n_d] \rightarrow F$$
	such that one of the two following conditions is met:
	\begin{enumerate}
		\item $f$ has image in $K$,
		\item the composite $$\square[n_1, \dotsc, n_d] \xrightarrow{f} F \rightarrow I$$ underlies a map of division pairs 
	$$(\boxdot[i], A[i]) \rightarrow (K,J)$$ for some injective map $i: \square[1,\dotsc,1] \rightarrow \square[n_1,\dotsc, n_d]$.
	\end{enumerate}
\end{definition}

\begin{example}
	The subgrid 
	\[
		F:= \begin{tikzcd}[column sep = 8, row sep = 8,execute at end picture={
			\scoped[on background layer]
			\fill[pamblue, opacity = 0.5] (a.center) -- (b.center) -- (d.center) -- (c.center);
		}]
			  |[alias=a]| (1,1)\arrow[r]\arrow[d] & (2,1) \arrow[r] &  |[alias=b]|(3,1)\arrow[d]\\
			  (1,2)\arrow[rr]\arrow[d] &  &  (3,2) \arrow[d] \\
			  |[alias=c]|(1,3) \arrow[rr] && |[alias=d]| (3,3)
		\end{tikzcd}
		\subset
		\begin{tikzcd}[column sep = 8, row sep = 8,execute at end picture={
			\scoped[on background layer]
			\fill[pamblue, opacity = 0.5] (a.center) -- (b.center) -- (d.center) -- (c.center);
		}]
		|[alias=a]|(0,0)\arrow[dd] \arrow[r, ""'name=f4] & (1,0) \arrow[d] \arrow[r, ""'name=f5] & (2,0)\arrow[d] \arrow[r] & |[alias=b]| (3,0)\arrow[d] \\
			& (1,1)\arrow[r]\arrow[d] & (2,1) \arrow[r] & (3,1)\arrow[d]\\
			(0,2)\arrow[d] \arrow[r,""name=t4] &(1,2)\arrow[d]\arrow[rr,""name=t6] &  & (3,2)\arrow[d] \\
			|[alias=c]| (0,3) \arrow[r] & (1,3) \arrow[rr] &  & |[alias=d]|(3,3)
		\end{tikzcd} =: I
	\]
	is fillable with respect to the division pair
	\[	
		(K,J) := \left( \begin{tikzcd}[column sep = 8, row sep = 8,execute at end picture={
			\scoped[on background layer]
			\fill[pamblue, opacity = 0.5] (a.center) -- (b.center) -- (d.center) -- (c.center);
			\scoped[on background layer]
			\fill[pamblue, opacity = 0.5] (e.center) -- (f.center) -- (g.center) -- (d.center);
		}]
		|[alias=a]|(0,0)\arrow[dd] \arrow[r, ""'name=f4] & |[alias=b]|(1,0) \arrow[d] \arrow[r, ""'name=f5] & (2,0) \arrow[r] & (3,0)\arrow[d] \\
			  & (1,1)\arrow[d] &  & (3,1)\arrow[d]\\
			 (0,2)\arrow[d] \arrow[r,""name=t4] &|[alias=e]|(1,2)\arrow[d]\arrow[rr,""name=t6] &  & |[alias=f]|(3,2)\arrow[d] \\
			 |[alias=c]| (0,3) \arrow[r] &|[alias=d]| (1,3) \arrow[rr] &  & |[alias=g]|(3,3),
		\end{tikzcd}\;\;
		\begin{tikzcd}[column sep = 8, row sep = 8,execute at end picture={
			\scoped[on background layer]
			\fill[pamblue, opacity = 0.5] (a.center) -- (b.center) -- (d.center) -- (c.center);
		}]
		|[alias=a]|(1,0) \arrow[d] \arrow[r, ""'name=f5] & (2,0)\arrow[d] \arrow[r] & |[alias=b]| (3,0)\arrow[d] \\
			(1,1) \arrow[r]\arrow[d] & (2,1) \arrow[r] & (3,1)\arrow[d]\\
			|[alias=c]| (1,2)\arrow[rr,""name=t6] &  & |[alias=d]|(3,2)
		\end{tikzcd}
		\right)
	\]
	for $I$. The unique injective map $\square[1,2] \rightarrow F$ underlies a map of division pairs $(\boxdot[i], A[i]) \rightarrow (K,J)$, where $i : \square[1,1] \rightarrow \square[1,2]$ selects 
	the top 2-box.

	In general, not each grid is necessarily fillable. Namely, the subgrid
	\[
		\begin{tikzcd}[column sep = 8, row sep = 8, execute at end picture={
			\scoped[on background layer]
			\fill[pamblue, opacity = 0.5] (a.center) -- (b.center) -- (d.center) -- (c.center);
		}]
			|[alias=a]|(0,0) \arrow[r]\arrow[d] & (0,1) \arrow[r] \arrow[d] &(2,0) \arrow[r] & |[alias=b]|(3,0)\arrow[d] \\
			(0,1) \arrow[r]\arrow[d] & (1,1) \arrow[d]\arrow[r] & (2,1) \arrow[r] & (3,1)\arrow[d] \\
			|[alias=c]|(0,2) \arrow[r] & (1,2) \arrow[rr] && |[alias=d]|(3,2)
		\end{tikzcd} \subset 
\begin{tikzcd}[column sep = 8, row sep = 8, execute at end picture={
	\scoped[on background layer]
	\fill[pamblue, opacity = 0.5] (a.center) -- (b.center) -- (d.center) -- (c.center);
}]
	|[alias=a]|(0,0) \arrow[r]\arrow[d] & (0,1) \arrow[r] \arrow[d] &(2,0)\arrow[d] \arrow[r] & |[alias=b]|(3,0)\arrow[d] \\
	(0,1) \arrow[r]\arrow[d] & (1,1) \arrow[d]\arrow[r] & (2,1) \arrow[r] & (3,1)\arrow[d] \\
	|[alias=c]|(0,2) \arrow[r] & (1,2) \arrow[rr] && |[alias=d]|(3,2)
\end{tikzcd}
\]
is \textit{not} fillable with respect to the division pair 
\[
	\left(\begin{tikzcd}[column sep = 8, row sep = 8, execute at end picture={
		\scoped[on background layer]
		\fill[pamblue, opacity = 0.5] (a.center) -- (b.center) -- (d.center) -- (c.center);
		\scoped[on background layer]
		\fill[pamblue, opacity = 0.5] (e.center) -- (f.center) -- (b.center) -- (g.center);
	}]
		(0,0) \arrow[r]\arrow[d] & (0,1) \arrow[r]  &  |[alias=e]| (2,0)\arrow[d] \arrow[r] & |[alias=f]|(3,0)\arrow[d] \\
		|[alias=a]| (0,1) \arrow[r]\arrow[d] & (1,1) \arrow[d]\arrow[r] &|[alias=g]|(2,1) \arrow[r] & |[alias=b]|(3,1)\arrow[d] \\
		|[alias=c]|(0,2) \arrow[r] & (1,2) \arrow[rr] && |[alias=d]|(3,2),
	\end{tikzcd}
	\begin{tikzcd}[column sep = 8, row sep = 8, execute at end picture={
		\scoped[on background layer]
		\fill[pamblue, opacity = 0.5] (a.center) -- (b.center) -- (d.center) -- (c.center);
	}]
		|[alias=a]|(0,0) \arrow[r]\arrow[d] & (0,1) \arrow[r] \arrow[d] &|[alias=b]|(2,0)\arrow[d] \\
		|[alias=c]|(0,1) \arrow[r] & (1,1) \arrow[r] & |[alias=d]|(2,1) 
	\end{tikzcd}
	\right).
\]	
\end{example}

The goal of this section is to prove the following theorem:

\begin{theorem}\label{thm.division-pair-fillables}
	Let $(K,J)$ be a division pair for $I$. Suppose that $F_1, \dotsc, F_n$ are $(K,J)$-fillable subshapes of $I$,
	then the inclusion 
	$$
	\textstyle [K] \cup [J] \cup \dtr_{\leq d-1}[I] \rightarrow \bigcup_{i=1}^n [F_i] \cup [K] \cup [J] \cup \dtr_{\leq d-1}[I]
	$$
	is a Segal equivalence.
\end{theorem}

Its demonstration requires some preparation. The following is one of the crucial observations:

\begin{proposition}\label{prop.division-pair-fillables-base}
	Suppose that $(K,J)$ is a division pair. Then for an $(K,J)$-fillable subshape $F$, the inclusion 
	$$[F\cap K] \cup [F \cap J] \cup \dtr_{\leq d-1}[F] \rightarrow [F]$$
	is a Segal equivalence.
\end{proposition}
\begin{proof}
	If $F$ is $(d-1)$-truncated or it satisfies condition (1) of 
	\ref{def.fillables}, this is clear, since then the left-hand side is the whole $[F]$ and the inclusion in question is the identity. Thus we may focus on the case that there exists a $d$-shaping map $f : \square[n_1, \dotsc, n_d] \rightarrow F$ (in particular, $f$ is injective) that underlies 
	a map of division pairs $$(\boxdot[i], A[i]) \rightarrow (K,J),$$
	for a suitable injective map $i$. 
	It then directly follows from the definitions that we obtain a pushout square 
	\[
		\begin{tikzcd}[column sep = small]
			{[\boxdot]} \cup [A] \cup \dtr_{\leq d-1}[\square] \arrow[r]\arrow[d] & {[F\cap K]} \cup [F\cap J] \cup \dtr_{\leq d-1}[F] \arrow[d] \\
			{[\square]} \arrow[r] & {[F]}.
		\end{tikzcd}
	\]
	Since the map on the left is a Segal equivalence on account of \ref{ex.division-pair-boxdot}, the map on the right must also be a Segal equivalence.
\end{proof}

The final ingredient for the proof of \ref{thm.division-pair-fillables}, is the following:

\begin{proposition}\label{prop.fillables-intersections}
	The $(K,J)$-fillables for a division pair $(K,J)$ are closed under taking intersections.
\end{proposition}

\begin{lemma}\label{lemma.pullback-standard-grids}
	Suppose that $I$ is a pasting shape and that we have injective maps $f : \square[n_1, \dotsc, n_d] \rightarrow I$ and $g : \square[m_1, \dotsc, m_d] \rightarrow I$ in $I$ with intersecting images, 
	then one can construct a pullback square 
	\[
		\begin{tikzcd}
			\square[l_1, \dotsc, l_d] \arrow[r]\arrow[d] & \square[n_1, \dotsc, n_d] \arrow[d,"f"] \\
			\square[m_1, \dotsc, m_d] \arrow[r, "g"] & I,
		\end{tikzcd}
	\]
	where all arrows appearing in the diagram are inclusions.
\end{lemma}
\begin{proof}
		Since the images of $f$ and $g$ intersect, we get that $\im(f_a) \cap \im(g_a)$ is non-empty for each index $a$. 
		Let $h_a : [l_a] \rightarrow \im(f_a) \cap \im(g_a)$ be the unique isomorphism for each $a$. These give 
		rise to maps $[l_a] \rightarrow \im(f_a) \cap \im(g_a) \rightarrow \im(f_a) \cong [n_a]$, inducing a map 
		$i : \square[l_1, \dotsc, l_d] \rightarrow \square[n_1, \dotsc, n_d]$. Similarly, we obtain a map 
		$j : \square[l_1, \dotsc, l_d] \rightarrow \square[n_1, \dotsc, n_d]$. It can be readily checked that this gives a commutative square 
		\[
			\begin{tikzcd}
				\square[l_1, \dotsc, l_d] \arrow[r, "i"]\arrow[d, "j"'] & \square[n_1, \dotsc, n_d] \arrow[d,"f"] \\
				\square[m_1, \dotsc, m_d] \arrow[r,"g"] & I.
			\end{tikzcd}
		\]
		To check that this square is a pullback square, it suffices to check this on the level of nerves as the nerve functor is fully faithful.
		Thus we must verify that for any $([k_1], \dotsc, [k_d]) \in \Delta^{\times d}$, the following 
		square is a pullback square 
		\[
			\begin{tikzcd}
				\prod_{1 \leq a \leq d} \Delta([k_a], [l_a]) \arrow[r, "\prod i_a \circ -"]\arrow[d, "\prod j_a \circ -"'] & \prod_{1 \leq a \leq d} \Delta([k_a], [n_a])\arrow[d, "\prod f_a \circ -"] \\
				 \prod_{1 \leq a \leq d} \Delta([k_a], [m_a])\arrow[r, "\prod g_a \circ -"] & \Hom_I(\coprod_{1 \leq a \leq d}[k_a], \N),
			\end{tikzcd}
		\]
		and this readily follows from the construction of the maps $i$ and $j$.
\end{proof}

\begin{lemma}\label{lemma.factored-maps}
	Let $(K,J)$ be a division pair for $I$. Suppose that the corners of $J$ are given by $(\alpha_1, \omega_1), \dotsc, (\alpha_d, \omega_d)$.
	Then for an injective map $f : \square[n_1, \dotsc, n_d] \rightarrow I$, the following are equivalent:
	\begin{enumerate}[noitemsep]
		\item $f$ underlies a map of division pairs 
		$$(\boxdot[i], A[i]) \rightarrow (K,J),$$
		\item for all $1 \leq a \leq d$, we have $\max(\min(f_a), \alpha_a) < \min(\max(f_a), \omega_a) \in \im(f_a)$.
	\end{enumerate}
\end{lemma}
\begin{proof}
	First, suppose that (1) is satisfied so that we have an injective map 
	$$
	i : \square[1,\dotsc, 1]\rightarrow \square[n_1,\dotsc, n_d]
	$$
	with the property that $fi$ has image in $J$ and $f$ carries every non-degenerate $d$-box of $\boxdot[i]$ to a non-degenerate 
	$d$-box of $K$.
	Define $\alpha_a' := f_a(i_a(0))$ and $\omega_a' := f_a(i_a(1))$. From \ref{prop.div-pair-boxes},
	it follows that 
	$\alpha_a \leq \alpha_a' < \omega_a' \leq \omega_a$.    
	If $f_a(0) < \alpha_a$, then the box $(x,y)$ given by $x_b = 0$ for all $b$ and $y_b = n_b$ if $a \neq b$ and $y_a = i_a(0)$, 
	is a non-degenerate $d$-box in $M^-_a \subset \boxdot$. 
	Hence $(f(x), f(y))$ is contained in $K$. Again by \ref{prop.div-pair-boxes},
	we necessarily must have that $f(y)_a = \alpha'_a \leq \alpha_a$.
 Hence $\alpha_a = \alpha_a'$.  Thus we deduce that $\max(f_a(0), \alpha_a) \in \im(f_a)$. 
	Similarly, 
	one deduces that $\min(f_a(n_a), \omega_a) \in \im(f_a)$. 

	Conversely, suppose that (2) is satisfied.
	Let $$i:\square[1,\dotsc,1] \rightarrow \square[n_1,\dotsc, n_d]$$ be the injective inclusion that satisfies $$f_a(i_a(0)) = \max(f_a(0), \alpha_a) <  \min(f_a(n_a),\omega_a) = f_a(i_a(1))$$ for all $a$. 
	If $(x,y)$ is a box in $\square$, then  $(f(x), f(y))$ is contained in $J$ if and only if $i_a(0) \leq x_a, y_a \leq i_a(1)$ for all $a$ on account of \ref{prop.div-pair-boxes}.
	Hence, $f^{-1}(J) = A[i]$. Let $(x,y)$ be a non-degenerate $d$-box of $\square$. Then we have to show 
	that $(f(x), f(y))$ is in $K$ if and only if $(x,y)$ is in $\boxdot$. If $(x,y) \in \boxdot$, then there 
	exists some index $a$ so that $x_a < y_a \leq i_a(0)$ (or $i_a(1) \leq x_a < y_a$, but this case is handled similarly). Thus $f_a(x_a) < f_a(y_a) \leq f_a(i_a(0))$. Hence, $f_a(i_a(0)) \neq f_a(0)$, so that we must have $f_a(i_a(0)) = \alpha_a$. Thus $(f(x), f(y)) \in K$. The converse implication is shown similarly, using \ref{prop.div-pair-boxes} and the assumption that the $f_a$'s are injective.
\end{proof}

\begin{corollary}\label{cor.shaping-fillable}
Let $(K,J)$ be a division pair  for $I$. Suppose that the corners of $J$ are given by $(\alpha_1, \omega_1), \dotsc, (\alpha_d, \omega_d)$. Let $F \subset I$ be a subshape that comes with a $d$-shaping map $f : \square[n_1, \dotsc, n_d] \rightarrow F$. 
Then $F$ is $(K,J)$-fillable if for each index $1 \leq a \leq d$, we have that $\max(\min(f_a), \alpha_a)$ and $\min(\max(f_a), \omega_a)$ are in the image of $f_a$.
\end{corollary}
\begin{proof}
Let us write $\alpha'_a := \max(\min(f_a), \alpha_a)$ and  $\omega_a' := \min(\max(f_a), \omega_a)$. 
If $\alpha'_a \geq \omega'_a$ for some $a$, then we must have that the image of $f$ is either in $K$, or $f_a(0) = f_a(n_a)$, in which case $F$ is $(d-1)$-truncated. Otherwise, 
we deduce that $f$ meets condition (2) of \ref{lemma.factored-maps}.
\end{proof}
\begin{proof}[Proof of \ref{prop.fillables-intersections}]	
	If $F$ and $F'$ have empty intersection, or one of these subshapes is $(d-1)$-truncated, then we are done. Consequently,
	we may assume there exist $d$-shaping maps $f : \square[n_1, \dotsc, n_d] \rightarrow I$ and $g : \square[m_1, \dotsc, m_d] \rightarrow I$ 
	that witness, respectively, $F$ and $F'$ to be $(K,J)$-fillable, and which have intersecting images.
	On account of \ref{lemma.pullback-standard-grids}, there is a pullback square 
	\[
		\begin{tikzcd}
			\square[l_1, \dotsc, l_d] \arrow[r]\arrow[d] & \square[n_1, \dotsc, n_d] \arrow[d] \\
			\square[m_1, \dotsc, m_d] \arrow[r] & I,
		\end{tikzcd}
	\]
	where all maps are injective.
	Note that this implies that the map
	$$
	h : \square[l_1,\dotsc, l_d] \rightarrow F \cap F'
	$$
	is $d$-shaping. Indeed, suppose that $\square[1,\dotsc, 1] \rightarrow F \cap F'$ is an injective map. Then 
	the composite $\square[1, \dotsc, 1] \rightarrow F \cap F' \rightarrow I$ factors through $f$ and $g$ by assumption, 
	yielding a commutative square
	\[
		\begin{tikzcd}
			\square[1,\dotsc,1] \arrow[r]\arrow[d] & \square[n_1,\dotsc,n_d] \arrow[d] \\
			\square[m_1, \dotsc, m_d] \arrow[r] & I,
		\end{tikzcd}
	\]
	thus the desired factorization $\square[1, \dotsc, 1] \rightarrow \square[l_1, \dotsc, l_d]$ exists in view of the universal property of the pullback.
	
	If $f$ or $g$ has image in $K$, then $h$ has so as well and we are done. Otherwise, condition (2) of \ref{lemma.factored-maps} is met 
	for both $f$ and $g$. Recall that by construction, 
	$$\im(h_a) = \im(f_a) \cap \im(g_a).$$
	One now readily checks that $\max(\min(h_a), \alpha_a)$ and  $\min(\max(h_a), \omega_a)$ are in the image of $h_a$ again, where the $(\alpha_a,\omega_a)$'s are as in the statement of the cited lemma.  
	Thus \ref{cor.shaping-fillable} now assures that $h$ again witnesses $F \cap F'$ to be $(K,J)$-fillable.
\end{proof}

We may now assemble our results to prove \ref{thm.division-pair-fillables}:

\begin{proof}[Proof of \ref{thm.division-pair-fillables}]
The inclusion $$\textstyle [K] \cup [J] \cup \dtr_{\leq d-1}[I] \rightarrow \bigcup_{i=1}^n [F_i] \cup [K] \cup [J] \cup \dtr_{\leq d-1} [I]$$ 
is a pushout of the map 
$$
\textstyle \bigcup_{i=1}^n [F_i \cap K] \cup [F_i \cap J] \cup {\dtr}_{\leq d-1}[F_i] \rightarrow \bigcup_{i=1}^n [F_i].
$$
Thus it suffices to show that the latter inclusion is a Segal equivalence.
This follows if we can show that the set $S$ of pairs of subsets of $[I]$ consisting of
$$
([F\cap K] \cup [F \cap J] \cup {\dtr}_{\leq d-1}[F], [F]),
$$
where $F \subset I$ is $(K,J)$-fillable,
 satisfies the conditions (1) and (2) of \ref{lemma.unions-seq}. But condition (2) is precisely \ref{prop.division-pair-fillables-base} and 
condition (1) follows from \ref{prop.fillables-intersections}.
\end{proof}

We conclude by demonstrating how \ref{cor.division-pair-good} follows from \ref{thm.division-pair-fillables}, which then completes the proof 
of all ingredients that went into the proof of pasting theorem \ref{thm.pasting-theorem}.

\begin{proof}[Proof of \ref{cor.division-pair-good}]
	Let $\mathscr{F}$ be the set of all subshapes of $I$ which are $(K,J)$-fillable. In light of \ref{prop.division-pair-spine} and \ref{thm.division-pair-fillables}, it is enough to show 
	that
	$$\textstyle [I] =\bigcup_{F \in \mathscr{F}} [F] \cup  [K] \cup [J]  \cup \dtr_{\leq d-1}[I].$$
	To this end, we have to show any injective map $f : \square[m_1, \dotsc, m_d] \rightarrow I$ with $m_1, \dotsc,m_d\neq 0$ has image 
	in either $K$, $J$, or a fillable $F \in \mathscr{F}$.
	By assumption, it suffices to check this for an extension $g : \square[n_1, \dotsc, n_d] \rightarrow I$ of $f$
	such that for each index $a$, $\im(f_a) \subset \im(g_a)$, and if 
	$g_a(0) \leq \alpha_a \leq g_a(n_a)$ or $g_a(0) \leq \omega_a \leq g_a(n_a)$ then, respectively, $\alpha_a \in \im(g_a)$ or $\omega_a \in \im(g_a)$. 
	If $\max(g_a(0), \alpha_a) \geq \min(g_a(n_a),\omega_a)$ for some index $a$, then $g$ must have image in $K$. 
	Hence, we may assume that $\max(g_a(0), \alpha_a) < \min(g_a(n_a),\omega_a)$ for all $a$. In this case, it follows from 
	the properties of the extension $g$ that $\max(g_a(0), \alpha_a), \min(g_a(n_a),\omega_a) \in \im(g_a)$. 
	Thus the image of $g$ is $(K,J)$-fillable in light of \ref{lemma.factored-maps}.
\end{proof}

\section{Outlook: an \texorpdfstring{$(\infty,d)$}{(∞,d)}-categorical pasting theorem}\label{section.outlook}
 
Recall that $d$-uple Segal spaces may be used to model $(\infty,d)$-categories. For $d=1$, this model is due to Rezk \cite{RezkSeg}, and the generalization 
to higher dimensions $d > 1$ is due to Barwick \cite{BarwickPhD}. The pasting theorem that we have proven in this article, may thus 
be specialized to a pasting theorem for $(\infty,d)$-categories. In this final section, we sketch an idea of how this can be achieved. 

Following \cite{Haugseng}, let us first 
recall Barwick's model of $(\infty,d)$-categories via $d$-uple Segal spaces. A feature of $d$-uple Segal spaces is its many different directions of cells: there 
are $d \choose k$ directions for $k$-cells. However, a $d$-uple Segal space that models an $(\infty,d)$-category should only have non-trivial 
$k$-cells in one preferred direction. This is captured in the following definition:

\begin{definition}
Every (1-uple) Segal space is by definition a \textit{1-fold Segal space}. Inductively, a $d$-uple Segal space $X$ 
is called a \textit{$d$-fold Segal space} if $X_{0, \bullet, \dotsc, \bullet}$ is essentially constant and additionally, $X_{1, \bullet, \dotsc, \bullet}$ is a $(d-1)$-fold Segal space. We denote 
the full subcategory of $\Cat^d(\S)$ spanned by the $d$-fold Segal space by $\Seg^d(\S)$.
\end{definition}

It follows that the non-trivial $k$-cells of a $d$-fold Segal space $X$ are concentrated in $X_{1,\dotsc, 1, 0, \dotsc, 0}$. By definition, equivalences 
in $\Seg^d(\S)$ are level-wise equivalences, and one can show that there are
 $d$-categorical equivalences (see \cite[Remark 7.19]{Haugseng}) between $d$-uple Segal spaces that are not level-wise equivalences. 
To remedy this, one needs to impose a further \textit{completeness} 
condition:

\begin{definition}
A Segal space $X$ is said to be \textit{complete} if the commutative square
\[
	\begin{tikzcd}
		X_0 \arrow[r]\arrow[d] & X_3 \arrow[d, "{(\{0 \leq 2\}^*, \{1\leq 3\}^*)}"] \\ 
		X_0 \times X_0 \arrow[r] & X_1 \times X_1
	\end{tikzcd}
\]
of spaces is a pullback square. Here the bottom and top arrows are degeneracy maps, and the left arrow is the diagonal. Inductively, a $d$-fold Segal space 
$X$ is called \textit{complete} if $X_{\bullet, 0, \dotsc, 0}$ and $X_{1, \bullet, \dotsc, \bullet}$ are complete $1$- and $(d-1)$-fold Segal spaces, respectively.
\end{definition}

\begin{remark}
	In the case that $d = 1$, the definition of complete Segal spaces is due to Rezk \cite{RezkSeg}.
	For $d > 1$, the definition of completeness coincides with Barwick's original criterion \cite{BarwickPhD} (see also \cite[Definition 14.1]{BarwickSchommerPries})  on account of \cite[Lemma 7.22]{Haugseng}.
\end{remark}

It is a celebrated result of Joyal and Tierney that complete Segal spaces  model  $\infty$-categories:

\begin{theorem}[\cite{JoyalTierney}]\label{thm.joyaltierney}
	The restriction of the Yoneda embedding $$\Cat_\infty \rightarrow \fun(\Delta^\op, \S) : \C \mapsto ([n] \mapsto \map_{\Cat_\infty}([n], \C))$$
	is an equivalence onto the full subcategory of $\fun(\Delta^\op, \S)$ spanned by the complete Segal spaces.
\end{theorem}

In \cite{BarwickSchommerPries}, Barwick and Schommer-Pries present an axiomatization of the $\infty$-category of $(\infty,d)$-categories and show 
that the full subcategory of $\Seg^d(\S)$ spanned by the complete $d$-fold Segal spaces, satisfies these axioms. Henceforth, we will 
write $\Cat_{(\infty,d)}$ for this subcategory. Direct comparisons with a selection of other models for $(\infty,d)$-categories can be found in \cite{BergnerRezk1}, \cite{BergnerRezk2}, and \cite{OzornovaRovelli} accompanied by \cite{Loubaton}. 

\begin{proposition}
	The inclusion of the full subcategory
	$$
	 \Seg^d(\S) \rightarrow \Cat^d(\S)
	$$
	admits a left adjoint. In other words, $\Seg^d(\S)$ is a reflective subcategory of $\Cat^d(\S)$.
\end{proposition}
\begin{proof}
	This follows from the same arguments as in the proof of\cite[Lemma 14.2]{BarwickSchommerPries}.
\end{proof}

We may now consider the composite functor 
$$
\{-\} : \Shape^d \xrightarrow{[-]} \Cat^d(\S) \rightarrow \Seg^d(\S),
$$
where the latter functor is the left adjoint to the inclusion $\Seg^d(\S) \subset \Cat^d(\S)$. In order to get a meaningful pasting theorem for $(\infty,d)$-categories, 
one should compute this functor. If $I$ is a $d$-dimensional pasting shape, we expect that the coherence data of a map of $d$-uple Segal spaces from the nerve $[I]$ to a  $d$-fold Segal space, 
is encoded by a complete and discrete $d$-fold Segal space associated to $I$. The simplices of this $d$-fold Segal space 
should admit an explicit description that is reminiscent of \ref{prop.simplices-nerve}. This would then imply the following:

\begin{conjecture}\label{conj.categorical-nerve}
	Let $I$ be a $d$-dimensional pasting shape. Then the $d$-categorical nerve $\{I\}$ of $I$ is a complete and discrete $d$-fold Segal space.
\end{conjecture}

\begin{remark}
	Complete and discrete $d$-fold Segal spaces correspond to \textit{gaunt $d$-categories} on account of \cite[Corollary 12.3]{BarwickSchommerPries}. Gaunt $d$-categories are $d$-categories in which every invertible cell is an identity, see also \cite[Section 2]{BarwickSchommerPries}. 
\end{remark}

\begin{remark}
	As a stepping stone towards proving \ref{conj.categorical-nerve}, one could first consider the case that $d=2$. Then 
	a $2$-fold Segal space is precisely a $2$-uple Segal space where all vertical arrows are degenerate. Pictorially, if one starts with a 2-dimensional pasting shape $I$ 
	that is represented by a graph as in \ref{ex.pshapes-graph}, then the 2-categorical nerve of $I$ would be the 2-category that is represented by the graph (with colorings) 
	that is obtained by contracting the vertical edges. For instance, one can compute that 
	\[
		\left\{\begin{tikzcd}[execute at end picture={
			\scoped[on background layer]
			\fill[pamblue, opacity = 0.5] (a.center) -- (b.center) -- (d.center) -- (c.center);
		}]
			|[alias=a]|(0,0) \arrow[rr]\arrow[d] && |[alias=b]|(2,0)\arrow[d]  \\
			(0,1) \arrow[r]\arrow[d] & (1,1) \arrow[d]\arrow[r] & (2,1)\arrow[d] \\
			|[alias=c]|(0,2) \arrow[r] & (1,2) \arrow[r] & |[alias=d]|(2,2)
		\end{tikzcd}\right\} \simeq 
		\begin{tikzcd}[column sep = large]
			0\arrow[rr, bend left = 50, ""'name=02]\arrow[r, bend right = 70, ""name=01_1] \arrow[r, ""'name=01_0] &|[""name=1]| 1\arrow[r, bend right = 70, ""name=12_1] \arrow[r, ""'name=12_0] & 2. 
			\arrow[from=02,to=1, end anchor={[yshift=-8pt]}, Rightarrow] 
			\arrow[from=01_0, to=01_1, Rightarrow]
			\arrow[from=12_0, to=12_1, Rightarrow]
		\end{tikzcd}
	\]
\end{remark}

Combining the conjecture with the pasting theorem we have proven in this article,  we would gain a meaningful pasting theorem for $(\infty,d)$-categories:

\begin{theorem}
	Suppose that \ref{conj.categorical-nerve} holds. 
	Let $I$ be a $d$-dimensional locally composable pasting shape covered by subshapes $I_1, \dotsc, I_n$. Then $I$ can be written as 
	$$
	I = \bigcup_{i=1}^n I_i,
	$$
	and this union is preserved by the $d$-categorical nerve functor so that the canonical map
	$$
\textstyle \colim_{J \in \mathcal{P}(\{I_1, \dotsc, I_n\})} \{J\} \rightarrow  \{I\}
$$
	is an equivalence of $d$-fold Segal spaces between complete and discrete $d$-fold Segal spaces. In particular, it is an equivalence of 
	$(\infty,d)$-categories.
\end{theorem}
\begin{proof}
	The main part follows from \ref{thm.pasting-theorem} and the fact that 
	left adjoints preserve colimits. The final assertion follows from the fact that the inclusion $\mathrm{Cat}_{(\infty,d)} \rightarrow \Seg^d(\S)$ is fully faithful, and hence 
	reflects colimits.
\end{proof}

\nocite{*}
\bibliographystyle{amsalpha}
\bibliography{Segal}

\end{document}